\documentclass[preprint,12pt]{elsarticle}
\usepackage{hyperref}
\usepackage{dsfont}
\usepackage{subcaption}
\usepackage{amsfonts,amsthm}

\newcommand{\dx}{{\mathrm d}}
\newcommand{\R}{{\mathbb R}}
\newcommand{\N}{{\mathbb N}}

\newcommand{\E}{{\mathbb E}}

\makeatletter
\def\ps@pprintTitle{%
  \let\@oddhead\@empty
  \let\@evenhead\@empty
  \let\@oddfoot\@empty
  \let\@evenfoot\@oddfoot
}
\makeatother

\newtheorem{Theorem}{Theorem}[section]
\newtheorem{Proposition}[Theorem]{Proposition}

\newtheorem{Definition}[Theorem]{Definition}
\newtheorem{Corollary}[Theorem]{Corollary}
\newtheorem{Lemma}[Theorem]{Lemma}
\newtheorem{Remark}[Theorem]{Remark}

\newtheorem{Assumption}[Theorem]{Assumption}

\usepackage{amssymb}
\usepackage{amsmath}
\begin{document}

\begin{frontmatter}

%% Title, authors and addresses

%% use the tnoteref command within \title for footnotes;
%% use the tnotetext command for theassociated footnote;
%% use the fnref command within \author or \affiliation for footnotes;
%% use the fntext command for theassociated footnote;
%% use the corref command within \author for corresponding author footnotes;
%% use the cortext command for theassociated footnote;
%% use the ead command for the email address,
%% and the form \ead[url] for the home page:
%% \title{Title\tnoteref{label1}}
%% \tnotetext[label1]{}
%% \author{Name\corref{cor1}\fnref{label2}}
%% \ead{email address}
%% \ead[url]{home page}
%% \fntext[label2]{}
%% \cortext[cor1]{}
%% \affiliation{organization={},
%%             addressline={},
%%             city={},
%%             postcode={},
%%             state={},
%%             country={}}
%% \fntext[label3]{}

\title{Functional approximation of the marked Hawkes risk process}

%% use optional labels to link authors explicitly to addresses:
%% \author[label1,label2]{}
%% \affiliation[label1]{organization={},
%%             addressline={},
%%             city={},
%%             postcode={},
%%             state={},
%%             country={}}
%%
%% \affiliation[label2]{organization={},
%%             addressline={},
%%             city={},
%%             postcode={},
%%             state={},
%%             country={}}

\author[1]{Mahmoud Khabou} %% Author name

%% Author affiliation
\affiliation[1]{organization={Imperial College London},%Department and Organization
            addressline={180 Queen's Gate, South Kensington}, 
            city={London},
            postcode={SW7 2AZ}, 
            %state={},
            country={United Kingdom}}

\author[2]{Laure Coutin} %% Author name

%% Author affiliation
\affiliation[2]{organization={Université Toulouse III - Paul Sabatier, IMT UMR CNRS 5219},%Department and Organization
            addressline={118 Route de Narbonne}, 
            city={Toulouse},
            postcode={31400}, 
            %state={},
            country={France}}

%% Abstract
\begin{abstract}
%% Text of abstract
The marked Hawkes risk process is a compound point process for which the occurrence and amplitude of past events impact the future. Thanks to its autoregressive properties, it found applications in various fields such as neuosciences, social networks and insurance.
Since data in real life is acquired over a discrete time grid, we propose a strong discrete-time approximation of the continuous-time Hawkes risk process obtained be embedding from the same Poisson measure. We then prove trajectorial convergence results both in some fractional Sobolev spaces and in the Skorokhod space, hence extending the theorems proven in the literature. We also provide upper bounds on the convergence speed with explicit dependence on the size of the discretisation step, the time horizon and the regularity of the kernel.
\end{abstract}

%%Graphical abstract
%\begin{graphicalabstract}
%\includegraphics{grabs}
%\end{graphicalabstract}

%%Research highlights
%\begin{highlights}
%\item Research highlight 1
%\item Research highlight 2
%\end{highlights}

%% Keywords
\begin{keyword}
Hawkes process \sep Euler scheme \sep Sobolev space \sep Skorokhod space
%% keywords here, in the form: keyword \sep keyword

%% PACS codes here, in the form: \PACS code \sep code

%% MSC codes here, in the form: \MSC code \sep code
%% or \MSC[2008] code \sep code (2000 is the default)

\end{keyword}

\end{frontmatter}

%% Add \usepackage{lineno} before \begin{document} and uncomment 
%% following line to enable line numbers
%% \linenumbers

%% main text
%%

%% Use \section commands to start a section
\section{Intoduction}
Initially introduced as a model for contagious events \cite{3b77a5cc-11e2-3e0a-b8e5-ecf028e7f216}, continuous-time linear Hawkes processes found applications in many fields involving self or cross excitation, such as portfolio credit risk \cite{errais2010affine}, microstructure price dynamics \cite{LEE2017154}, social media networks \cite{LOUZADAPINTO201686} and earthquakes \cite{doi:10.1080/01621459.1988.10478560}. The Hawkes process was then extended to the nonlinear setting in the seminal paper \cite{BM}, allowing thus for a more general dependence on the past including self-inhibition, in opposition to the mere affine dependence allowed by linear process. This generalization comes at a price: the nonlinear Hawkes process lacks the Galton-Watson framework as well as closed formulae for its expected value and covariance \cite{hillairet2023explicit}. Nonetheless, nonlinear Hawkes processes have found applications in fields where self and cross inhibition are crucial, like neuroscience \cite{LAMBERT20189}.\\
As the quantity of data involving bin counts (or count series); that is series representing the number of events observed on regularly spaced time intervals, has been increasing, appropriate discrete-time models became a relevant object of research. In the context of Hawkes processes, a number of count series models has been proposed to capture the self exciting (or inhibiting) aspects of a given dynamics. One of them is the integer auto-regressive process of order $p$ (INAR($p$)) \cite{d5dd0536-ae69-3f80-a080-f512bf5a819d}, whose generalization of the infinity order (INAR($\infty$)) has been proven to be a discrete-time version of linear Hawkes processes (\textit{cf.} \cite{kirchner} and the convergence results therein). In the nonlinear case, the Markovian Poisson autoregressions studied in \cite{fokianos2012nonlinear} can be seen as discrete-time version of nonlinear Hawkes processes with exponential kernels, despite the connection between the two processes not being stated in that article. As a generalization, a class of Poisson autoregressions is proposed in \cite{mahmoud}. This class of processes is then proven to be a weak approximation of Hawkes processes with Erlang kernels (that is the product of a polynomial and an exponential) in the Skorokhod topology. In this article, we suggest a straightforward approach based on the intuitive discretisation of a stochastic integral. To give an explicit illustration, we start by recalling the definition of a continuous-time Hawkes process. Let $N=(N_t)_{t\in [0,T]}$ be a point process observed on a time interval $[0,T]$ and measurable with respect to its canonical filtration $\mathcal F^N$. Assume that $N$ has an intensity, that is a predictable process measuring the propensity of $N$ to jump in the near future, or informally 
\begin{equation}
    \label{eq:informal}
    \lambda_t \dx t = \mathbb E \left[ \dx N_t | \mathcal F^N_{t-}\right ],
\end{equation}
where $\dx N_t= N_{t+\dx t} -N_t$  takes the values $1$ or $0$, depending on the presence of a jump at time $t$. Obviously, the larger $\lambda_t$ is at a given time $t$, the more likely $N_t$ is to jump and vice versa.\\

We say that $N$ is a Hawkes process of kernel $h$ and jump-rate $\psi$ if the intensity takes the form 
\begin{equation}
    \label{eq:intensity}
    \lambda_t=\psi \left(\mu+\int_0^{t-}h(t-s) \dx N_s\right)\quad t\in [0,T]
\end{equation}
where the integral is taken in the Stieltjes sense. We take $\mu$ a constant playing the role of a baseline level, $\psi$  positive Lipschitz and $h$  integrable. Up to some standard stability hypothesis that will be stated in the next section, it is possible to build the Hawkes process on $\mathbb R_+$ and show that, just like the Poisson process, it is of order $O(T)$ on average. \\
In this article, we work with a discrete-time model entirely based on the Riemann sum approximation of integral \eqref{eq:intensity}. Indeed, given a discretisation time step $\Delta > 0$, the infinitesimal increment $\dx N_t$ can be seen as the number of events observed in the "small" time interval $\left( n \Delta,(n+1) \Delta  \right]$ where $n=\lfloor \frac{t}{\Delta} \rfloor$. 
Giving a discrete approximation of $N$ is then equivalent to giving a sequence of bin counts $(X_0, X_1,\cdots, X_M)$, where $M= \lfloor \frac{T}{\Delta} \rfloor$. Knowing all of the past count values, the count series $X_n$ is simulated according to a Poisson distribution of parameter $\Delta l_n^\Delta$ where
 the expression of the (discrete) intensity is 
$$l^\Delta_n=\psi \left( \mu + \sum _{k=1}^{n-1}h\left(\Delta(n-k)\right) X_k\right),$$
ensuring that it is predictable with respect to the filtration generated by $(X_0,\cdots, X_n)$.\\

The choice of the Poisson distribution $\mathcal P (\Delta l^\Delta_n)$ guarantees that $\Delta l ^\Delta_n= \E [X_n| X_0,\cdots,X_{n-1}]$ (which is the discrete equivalent of $\eqref{eq:informal}$) and is preferred to the more trivial choice of the uniform distribution $\mathcal U (\Delta l ^\Delta _n)$ used in \cite{SEOL2015223} for the following reasons: 
\begin{enumerate}
    \item A priori, we do not have a guarantee that $\Delta l^\Delta_n <1$ .
    \item Even for reasonably small time-steps $\Delta$, one can always expect to see two or more events in a given time bin, especially if events are clustered due to self-excitation. 
\end{enumerate}
We then embed the discrete time process back into the continuous time setting by means of càdlàg (right continuous with left limits) embedding, thus obtaining a new process $N^\Delta$ defined as $$N^\Delta_t= \sum_{k=0}^{\lfloor t/\Delta\rfloor }X_k.$$  Since the new proccess  has a discontinuous trajectory, the uniform distance fails to capture its proximity to the original process $N$, as two very similar trajectories will have a large distance in the uniform metric as soon as they have one discontinuity that does not take place at the exact same time. This is why, we provide approximation results in two different spaces that are more adapted to càdlàg processes: the fractional Sobolev space and the Skorokhod space, both of which will be defined in Section \ref{sec:main_result}. \\

To the best of our knowledge, this is the first work that provides strong approximation results for Hawkes processes, the other two \cite{mahmoud,kirchner} proving weak convergence in the Skorokhod metric. Furthermore, the bounds are provided in this article with explicit dependence on the time step $\Delta$ and the time horizon $T$, making them useful for numerical applications. \\
It is also worthwhile to highlight that our framework is different from the classical approximation results of Poisson jump-diffusion SDEs, because inter-arrival times cannot be explicitly simulated for Hawkes processes (unless the process is linear with an exponential kernel \cite{DZ}. \\

The article is organized as follows: In Section \ref{sec:def} we give the rigorous definitions of the continuous-time and discrete-time Hawkes processes, as results of thinning from the same underlying Poisson randomness. In Section \ref{sec:main_result}, explicit bounds on the Sobolev and Skorokhod distances between the continuous-time Hawkes process and its discrete-time approximation are given. Sections \ref{sec:appendix} and \ref{sec:lemmata} contain the different proofs that were necessary to obtain the main results. 

\section{Definitions}
\label{sec:def}
In the two following subsections, we define the compound marked Hawkes process, both in continuous time and discrete time using the Poisson embedding idea of \cite{BM}. The main source of randomness for these two processes is a tri-dimensional Poisson measure $P$ that takes values in the configuration space 
$$\Omega := \left \{ \omega = \sum_{i=1}^{n} \delta_{(\tau'_i, \theta_i, y_i)}, 0 < \tau'_1 < \cdots <\tau'_n, (\theta_i,y_i) \in \R_+ \times \R \text { and } n \in \mathbb N \cup \{+\infty \}\right \}.$$
Given a non negative Borel measure $\nu$ on $\R$ such that $\nu (\R)=1$, we take $\mathbb P $ to be the probability measure under which the point measure $P$ defined as 
$$P\left((0,t], (0,\theta], (-\infty,y] \right)(\omega) = \omega ((0,t], (0,\theta], (-\infty,y]), \quad (t,\theta,y) \in \R_+^2 \times \R$$
is a Poisson measure with intensity $\dx t \times \dx \theta \times \nu(\dx y) $. We also let $\mathcal F_t=\sigma \left(P(\mathcal T \times  S), \mathcal T \subset \mathcal B ((0,t]), S \in \mathcal B (\R_+ \times \R) \right)$
be the filtration associated with $P$. Throughout this paper, the conditional expectation knowing $\mathcal F_t$ is denoted by $\E_t$. \\
\subsection{The continuous time setting}
For a fixed $T>0$, let $h$ be a function in $L_1([0,T])$. Let $\psi$ be a positive $L-$Lipschitz function on $\R$ and $b$ be a positive Borel function on $\R$. We now state the stability assumption on the aforementioned elements.
\begin{Assumption}
\label{ass:stability}
We have that 
$$\rho_h:=L\|h\|_1\E[b(Y)]<1$$
where $Y$ is a random variable of distribution $\nu$ and $\|h\|_1= \int_0^T |h(t) |\dx t.$ We also assume that $\nu$ has a finite first moment, that is $\E |Y| <+\infty$.

\end{Assumption}

We now give the definition of the marked 
 compound Hawkes process (or risk) as a result of thinning from the underlying tri-dimensional Poisson measure $P$. The procedure is now standard (see for example \cite{BM} or \cite{1056305}) and gives the Hawkes process as the unique pathwise solution of a stochastic differential equation (SDE) examined in more detail in \cite{HILLAIRET202389}. 
 \begin{Definition}
 \label{def:hawkes_continuous}
 Let $P$ be a tri-dimensional Poisson measure on $\R_+^2 \times \R$ of intensity $\dx t \dx\theta \nu(\dx y)$.
     Fix $T>0$  a time horizon and let $h \in L_1 ([0,T])$, $\psi: \R \to \R_+$ 
 $L-$Lipschitz and $b:\R \to \R_+$ such that Assumption  \ref{ass:stability} is in force. The SDE on $[0,T]$
\begin{equation*}
    \begin{cases}
        R_t &=\int_{(0,t] \times \R_+ \times \R} y\mathds 1_{\theta \leq \lambda _s} P(\dx s, \dx \theta, \dx y)\\
        \lambda_t &=\psi \left( \int_{[0,t) \times \R_+ \times \R} h(t-s) \mathds 1_{\theta \leq \lambda _s} b(y)P(\dx s, \dx \theta, \dx y) \right)
    \end{cases}
\end{equation*}
has a unique pathwise solution $(R, \lambda)$ such that $R$ is $\mathcal F$ measurable and $\lambda$ is $\mathcal F$ predictable. \\
We say that $R$ is a marked Hawkes risk of kernel $h$, jump-rate $\psi$ and claim size distribution $\nu$. We call $\lambda$ the intensity of $R$ and $b$ the marks modulation function. \\
Furthermore, we define the marked simple Hawkes process $N$ as
$$N_t :=\int_{(0,t] \times \R_+ \times \R} \mathds 1_{\theta \leq \lambda _s} P(\dx s, \dx \theta, \dx y),~~t\in [0,T]$$
and the auxiliary process as 
$$\xi_t:=\int_{(0,t] \times \R_+ \times \R} b(y)\mathds 1_{\theta \leq \lambda _s} P(\dx s, \dx \theta, \dx y).$$
 \end{Definition}
 \begin{Remark}
     The assumption that $b$ is positive is superfluous from a mathematical point of view and can be omitted (up to the introduction of absolute values). However, we chose to keep it here because we want the excitation/inhibition to be determined by the sign of the kernel $h$, assuming that $\psi$ is monotonous. \\
     
 \end{Remark}

     If $\psi(x)=\mu+ x, x \in \R_+$ for a positive constant $\mu$  under the constraint $h\geq 0$ we say that the Hawkes process/risk is \textit{linear}. In this particular case, the Hawkes dynamics have a branching process representation.  Despite the fact that the first two moments of the process are explicitly known (up to the computation of an infinite sum of convolutions of $h$, \textit{cf.} \cite{BACRY20132475} and \cite{HILLAIRET202389}), linear Hawkes processes do not allow for self-inhibition. \\
 
 The marked Hawkes risk can also be defined in a more elementary (yet informal) way without using the thinning procedure. Let $N$ be a simple point process on $[0,T]$ and $(Y_k)_{k\in \mathbb N}$ a family of \textit{iid} random variables of common distribution $\nu$. Note that, we do not assume the variables $(Y_k)_{k\in \mathbb N}$ and the process $N$ to be independent. $N$ is said to be a marked simple Hawkes process if its intensity $\lambda$ follows the dynamics
 $$\lambda_t = \psi \left( \sum_{\tau_i<t} h(t-\tau_i) b(Y_i)\right),$$
 where $(\tau_i)_{i\in \mathbb N}$ are the arrival times of the points of $N$. Note that the intensity can also be put under the integral form 
 $$\lambda_t = \psi \left( \int_{0}^{t-} h(t-s) \dx \xi_s\right),$$
 where $\xi$ is the auxiliary process $\xi_t=\sum_{k=1}^{N_t}b(Y_k),,~~t\in [0,T]$.\\
 The risk process is simply the aggregation of all the marks
 $$R_t=\sum_{k=1}^{N_t} Y_k,,~~t\in [0,T].$$

 The marked Hawkes process is useful in modelling phenomena where the intensity is not only impacted by the realisation of an event $\tau_i$, but also by its "severity" $Y_i$. For instance, the choice $b(y)= \mathds 1 _{y\geq a}$ means that only claims of a size larger than a given threshold $a$ have an impact on the intensity. Karabash and Zhu \cite{karabash} provided limit theorems for a general class of marked Hawkes processes, albeit in the linear setting.\\
 This constitutes a generalization of the compound Hawkes process usually studied in the literature (\textit{cf.} \cite{errais2010affine}, \cite{khabou2023normal}) where the marks $Y$ do not impact the intensity.\\
 If the modulation function $b$ is chosen to be equal to the constant $1$, we can retrieve the usual unmarked intensity $\lambda_t=\psi\left(\int_0^{t-}h(t-s) \dx N_s \right)$. Similarly, the choice $Y \equiv 1$ ensures that $R\equiv N$, hence we will focus exclusively on $R$. \\

The goal is to suggest an intuitive discretization scheme on $[0,T]$ and to yield a bound on the distance between this scheme and the continuous time process in a convenient functional space.

%\begin{Remark}
%    If $h=\tilde{h}$ almost surely, then $(N,R,\lambda)$ and $(\tilde{N},\tilde{R},\tilde{\lambda})$ are indistinguishable (see Lemma \ref{lmm:laure}). 
 %   For instance if, $h(s) ={\mathbf 1}_{ {\mathbf Q}}$ and $b=1,$ $N$ is a Poisson process with intensity $\mu$ and we choose $h=0.$ In the sequel, we always choose the more regular version of $h.$ 
  %  \end{Remark}
\subsection{The discrete time setting}
%Throughout this paper, we assume that the time interval $[0,T]$ is discretized into $M \in \mathbb N$ intervals of equal size $\Delta= T/M$. The discretisation 
Throughout this paper, the bounded interval $[0,T]$ is discretised into $M\in \N^*$ equidistant intervals $(t_i,t_{i+1}]$ of length $\Delta$, where $\Delta=T/M$. For a given $t\in [0,T]$, we define $ (t)_ \Delta = \lfloor t/\Delta \rfloor \Delta $ to be its projection on the time grid. We also define $n_t=\lfloor t/\Delta \rfloor.$\\
For a $k\leq M$, we set $h_k=h(k\Delta)$. We omit the dependence on $\Delta$ to avoid cumbersome notation. Before defining the discrete-time marked Hawkes risk, we give the following stability assumption
\begin{Assumption}
    \label{ass:stability_discrete}
    Assume that $$\rho_{h,\Delta}:=L \sum_{k=1}^M |h_k| \Delta \E [b(Y)]<1,$$
    where $L$ is the Lipschitz coefficient of $\psi$ and $Y$ is a random variable of distribution $ \nu$.  We also assume that $\nu$ has a finite first moment, that is $\E |Y| <+\infty$.
\end{Assumption}
Just like the continuous-time case, we build the discrete-time marked Hawkes risk using the same tri-dimensional Poisson measure $P$. 

\begin{Definition}
    \label{def:hawkes_discrete}
     Let $P$ be a tri-dimensional Poisson measure on $\R_+^2 \times \R$ of intensity $\dx t \dx\theta \nu(\dx y)$.
     Fix $T>0$  a time horizon and let $h \in L_1 ([0,T])$, $\psi: \R \to \R_+$ 
 $L-$Lipschitz and $b:\R \to \R_+$ such that Assumption  \ref{ass:stability_discrete} is in force. Fix $M \in \mathbb N$ and let $\Delta= T/M$. \\
 Define the measurable (with respect to the filtration $(\mathcal F_{n\Delta})_{n=0,\cdots,M}$) sequence $(X^\Delta_k)_{k=0,\cdots,M}$ and predictable sequence $(l^\Delta_k)_{k=0,\cdots,M}$ recursively
 \begin{equation}
 \label{eq:def_discrete}
     \begin{cases}
         X^\Delta_0&=0, ~~
         l^\Delta_0=l^\Delta_1=\psi(0), ~~ D^{\Delta}_0=0,\\
         X^\Delta_n&=\int_{((n-1)\Delta,n\Delta]\times \R_+ \times \R} b(y) \mathds 1_{\theta \leq l^\Delta_n} P (\dx s , \dx \theta, \dx y ),\\
         l^\Delta_n&=\psi \left (\sum _{k=1}^{n-1} h_{n-k}X_k\right),\\
         D^{\Delta}_n&= \int_{((n-1)\Delta,n\Delta]\times \R_+ \times \R} \mathds 1_{\theta \leq l^\Delta_n} P (\dx s , \dx \theta, \dx y )
     \end{cases}
     .
 \end{equation}
 The discrete-time Hawkes risk $R^\Delta$ and intensity $\lambda^\Delta$ are the càdlàg piecewise constant processes defined as
 \begin{equation*}
     \begin{cases}
         \lambda^\Delta_t&=\lambda^\Delta_{(t)_\Delta}=l^\Delta_{n_t}\\
         R^\Delta_t&=R^\Delta_{(t)_\Delta}=\sum_{k=1}^{n_t}\int_{((k-1)\Delta,k\Delta]\times \R_+ \times \R} y \mathds 1_{\theta \leq \lambda^\Delta_{k\Delta}}P(\dx s, \dx \theta, \dx y) 
     \end{cases}
     .
 \end{equation*}
 Furthermore, the discrete-time auxiliary process can be defined as 
 $$\xi^\Delta_t=\xi^\Delta_{(t)_\Delta}=\sum_{k=1}^{n_t}X_k$$
 and the discrete Hawkes process as 
 $$N^{\Delta}_t=N^\Delta_{(t)_\Delta}=\sum_{k=1}^{n_t}D_k$$
\end{Definition}
Note that, despite their names, the discrete-time processes defined above are continuous time embeddings of the times series $l^\Delta$ and $X^\Delta$. However, their values are allowed to change exclusively on the points of the discretisation grid. \\
For simulation purposes, it is possible to build the sequences $X^\Delta$ and $l^\Delta$ without simulating the underlying Poisson measure $P$. This is based on the following observation: knowing $X^\Delta_1, \cdots, X^\Delta_n$ (and thus $l^\Delta_{n+1}$ according to equation \eqref{eq:def_discrete}) the variable 
$$X^\Delta_{n+1}=\int_{(n\Delta,(n+1)\Delta]\times \R_+ \times \R} b(y) \mathds 1_{\theta \leq l^\Delta_{n+1}} P (\dx s , \dx \theta, \dx y )$$
is a compound Poisson variable, that is 
$$X^\Delta_{n+1}=\sum_{k=1}^{D^\Delta}b(Y_k),$$
where $D^\Delta | (X_1,\cdots,X_n) \sim \mathcal P (\Delta l^\Delta_{n+1})$ and $Y_1,\cdots,Y_{D^\Delta}$ are \textit{iid} variables of distribution $\nu$.\\
Finally, the risk process can be recursively constructed on the time grid 
$$R^\Delta_{(n+1)\Delta}=R^\Delta_{n\Delta}+ \sum_{k=1}^{D^\Delta} Y_k, $$
where $Y_1,\cdots,Y_{D^\Delta}$ are the same variables used in the computation of $X^\Delta_{n+1}$. \\
Intuitively, the variable $X^\Delta_{n+1}=\sum_{k=1}^{D^\Delta}b(Y_k)$ is the discrete equivalent of the increment $\dx \xi_t$ defined in \ref{def:hawkes_continuous}. This is why we see $R^\Delta$ as a good approximation of $R$, which is what we will prove in the rest of the paper.\\
We conclude this subsection by discussing the numerical cost of simulating the discrete time process. Generally, the computation of the recursion \eqref{eq:def_discrete} is of order $O(M^2)$. This cost can be reduced in two cases:
\begin{enumerate}
    \item If the kernel is an Erlang function, that is $h(t)=Q(t)e^{-\beta t}$ where $\beta >0$ and $Q$ is a polynomial of degree $q$, then  the intensity is a Markov chain (up to the introduction of auxiliary processes, \textit{cf.} \cite{mahmoud}) and the computation cost is of order $O(qM)$.
    \item If the kernel $h$ is of compact support, that is $h(t)=0, \forall t \geq 
 S$, then the cost is of order $O(rM)$, where $r=S/\Delta$. 
\end{enumerate}
Before we show that $R^\Delta$ converges to $R$ as the time step $\Delta$ goes to zero, we motivate the choice of the spaces in which the convergence takes place.

\begin{Remark}
    We seek to approximate $h$ by a piecewise constant function. We choose to approximate $h$ on $[t_k,t_{k+1}[$ by $h(t_k).$ We can also approximate $h$ on $[t_k,t_{k+1}[$ by $ \Delta^{-1} \int_{(t_{k},t_{k+1}]} h(s)  \mathrm d s.$ This latter choice is more comfortable from a mathematical point of view.
    %In this setting, the kernel $s\mapsto c\frac{1}{\sqrt{s}} e^{-s}$ for suitable constant $c$  can be handle. 
    Nevertheless, the exact computation of the coefficients $(\Delta^{-1} \int_{(t_{k},t_{k+1}]} h(s)  \mathrm d s)_k$ can be cumbersome for general kernel $h$. 
    %\textcolor{magenta}{a traduire}
    %Notre schéma d'approximation peut touefois être utilisé pour un tel noyau en rajoutant une étape supplementaire. Si $h$ est un noyau general il est approximable dans $L^1$ par un noyau continu par morceau $h_{\varepsilon}$ avec $\|(h-h_{\varepsilon}{\mathbf 1}_{[0,T]}\leq \varepsilon.$ Ensuite, on approxime le processus de Hawkes $N^{\varepsilon}$ associé au noyau $h_{\varepsilon}$ par un processus de Hawkes discret selon la procédure d\'ecrite dans la section ???? $N^{\Delta,\varepsilon}.$ L'erreur commise en approximant le processus $N$ par le processus $N^{\Delta,\varepsilon}$ s'obtient alors en cumulant les erreurs obtenues dans in Lemma \ref{lmm:laure}  et le th\'eorème ???.
\end{Remark}

%\textcolor{blue}{J'ai mis la remarque en commentaire}
%\begin{Remark}
 % \textcolor{magenta}{We will precise the effect of the regularity of $h$ on  the rate of convergence in $\Delta$.  }  
%\end{Remark}
%{\bf Youssoufa et Sacha avaient quelques resultats sur les hawkes discrets : est-ce qu'on leu consacre un paragraphe ?}
Before we show that $R^\Delta$ converges to $R$ as the time step $\Delta$ goes to zero, we motivate the choice of the spaces in which the convergence takes place.
\subsection{Discussion on functional spaces}
%\textcolor{blue}{Je supprime la phrase 'in the literature' car tout ce qu je trouve est adapté aux browniens et donc est avec un carré}\\
The quality of the strong approximation of continuous stochastic processes $(Z_t)_{t\in [0,T]}$ (and in particular, diffusion SDEs) is evaluated by controlling the average uniform error 
$$\E \left [\sup_{0\leq t \leq T}|Z_t-Z^\Delta_t| \right].$$
While supremum norm is adapted to processes that have continuous trajectories, it can also be extended to  jump-diffusion SDEs driven by a homogeneous Poisson process \cite{BRUTILIBERATI2007982}. This is the case because the arrival times of a Poisson process $(\tau_i)$ are explicitly known and can be simulated exactly (unlike for general Hawkes processes), therefore rendering the problem equivalent to approximating a diffusion on $(\tau_i, \tau_{i+1}]$.  \\
The supremum norm however, should not be used to evaluate the proximity of càdlàg trajectories. 

% We seek to investigate the convergence of the discrete Hawkes process towards the Hawkes process in a suitable functional space. Several choices are possible. 

By construction, the Hawkes process and the discrete Hawkes process are càdlàg  picewise constant functions.
The  Hawkes process jumps time are contained in the set of the jumps time of the underlying Poisson random measure. On the other hand, the jumps of the discrete Hawkes process are contained in $\{t_k,~~k=1,...,N\}$ the set of the points of the subdivision. Thus, almost surely $\sup_{s\in [0,T]} | N_s - N^{\Delta}_s|{\geq} 1$ on a set with non null probability.
The processes $(N^{\Delta})_{\Delta \in (0,1)}$ will not converge to $N$ in the uniform norm on compact sets of ${\mathbb R}^+$ almost surely. 
This leads us to compute the distance between the Hawkes process and the discrete Hawkes process in ${\mathbb D}([0,T],{\mathbb R})$ equipped with the Skorokhod metric
\begin{equation}
\label{def:Skorokhod}
d_S(f,g)=\inf_{\mu \in \Lambda} \left \{ \sup_{0\leq t \leq T} |t-\mu(t)| \vee \sup _{0\leq t \leq T} |f(t)-g(\mu(t))| \right \},
\end{equation}
where $\Lambda$ is the set of strictly increasing continuous functions from $[0,T]$ to itself such that $\mu(0)=0$ and $\mu(T)=T$.
Note that, since this distance allows for some flexibility in the time at which the jump takes place (the role of the time change $\mu$), the problem that occurs with the uniform distance disappears. 

{By working in ${\mathbb D}([0,T],{\mathbb R})$ we do not pay attention to the properties  or the potential regularity of the paths of the processes.}
As increasing processes, the sample path of the Hawkes process and the discrete Hawkes process belong to set of function with finite bounded variation starting from 0 endowed with the distance
\begin{align*}
 d_{FV,T}(f,g)=    \sup_{(s_i)}\sum_{i}|f(s_i) - g(s_i)|
\end{align*}
where the suppremum run over all finite partitions $(s_i)$ of $[0,T].$ 
One should notice that $\sup_{s\in [0,T]}|f(s) -g(s)| \leq  d_{FV,T}(f,g)$ if $f(0)=g(0)=0.$ Thus, $N^{\Delta}$ will not converges to $N$ when $\Delta$ converges to $0$ for the  $ d_{FV,T}$ distance. 
%Let $BV$ be the space of bounded variation functions, endowed with the norm $\|f\|_{BV,T}=\int_0^T |f(s)| ds +  d_{FV,T}(f,0).$ 
In \cite{junca},  {the authors} %Christian Bourdarias, Marguerite Gisclon, St\'ephane Junca 
prove that the space of bounded variation is continuously embedding in  $\cap_{\eta<1}W^{\eta,1}$. {The fractional Sobolev spaces are interpolated spaces between $L^1([0,T])$ and the classical Sobolev space $W^{1,1}.$  Moreover, since, $R$ and $R^{\Delta}$ are linear combination of indicator functions ${\mathbf 1}_{[a,+\infty[},$
they belong to some suitable  Riemann-Liouville Fractional Sobolev spaces.
Thus, we compute the distance between the Hawkes process and the discrete Hawkes process in $W^{\eta,1}$ for all $0<\eta <1$ and in  Riemann-Liouville Fractional Sobolev spaces. See \cite{BLNT} for some definitions.}

\section{Main results}
\label{sec:main_result}
In the following section, we show the strong convergence of the discrete-time Hawkes risk process to the continuous-time counterpart and give convergence rates both in the time-step $\Delta$ and the time horizon $T$. 
In this section $K$ is a constant which may depend on $L,$ $\nu,$  $h$ and  $b$ but is independent of $T$ and $\Delta.$ We emphasize that while $K$ is finite if the stability assumption \ref{ass:stability} holds, it can diverge to infinity as $\rho_h$ approaches 1. \\ 
We also emphasize that our continuous and discrete time processes are thinned from the same Poisson measure $P$. As an illustration, the following figure shows both processes and their underlying common randomness for two values of the time step $\Delta$:

\begin{figure*}[h!]
    \centering
    \begin{subfigure}[t]{0.5\textwidth}
        \centering
        \includegraphics[height=50mm]{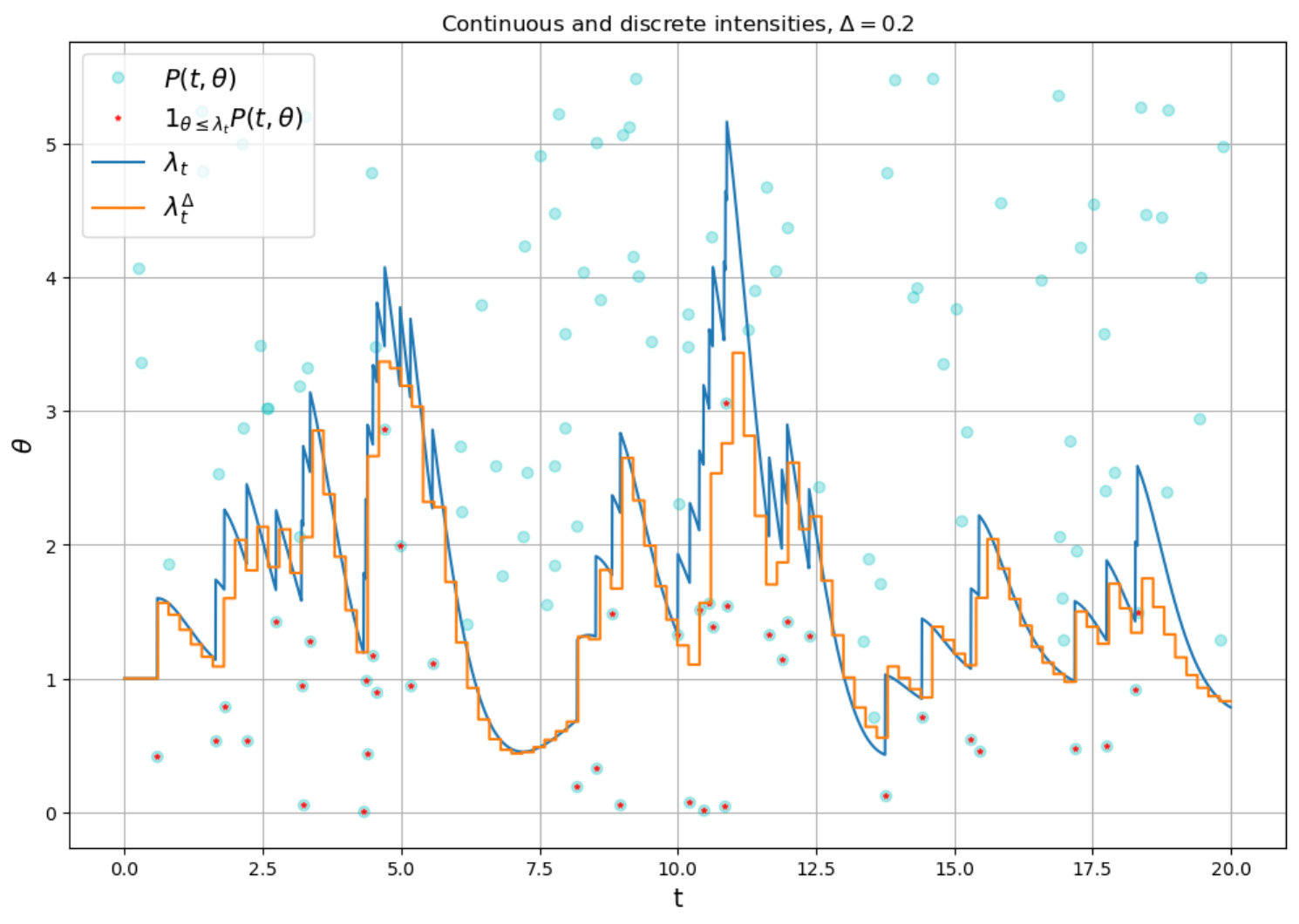}
        \caption{When the discretisation step $\Delta$ is relatively large, the discrete intensity is more susceptible to miss points that are accepted by the continuous time trajectory.}
    \end{subfigure}%
    ~ 
    \begin{subfigure}[t]{0.5\textwidth}
        \centering
        \includegraphics[height=50mm]{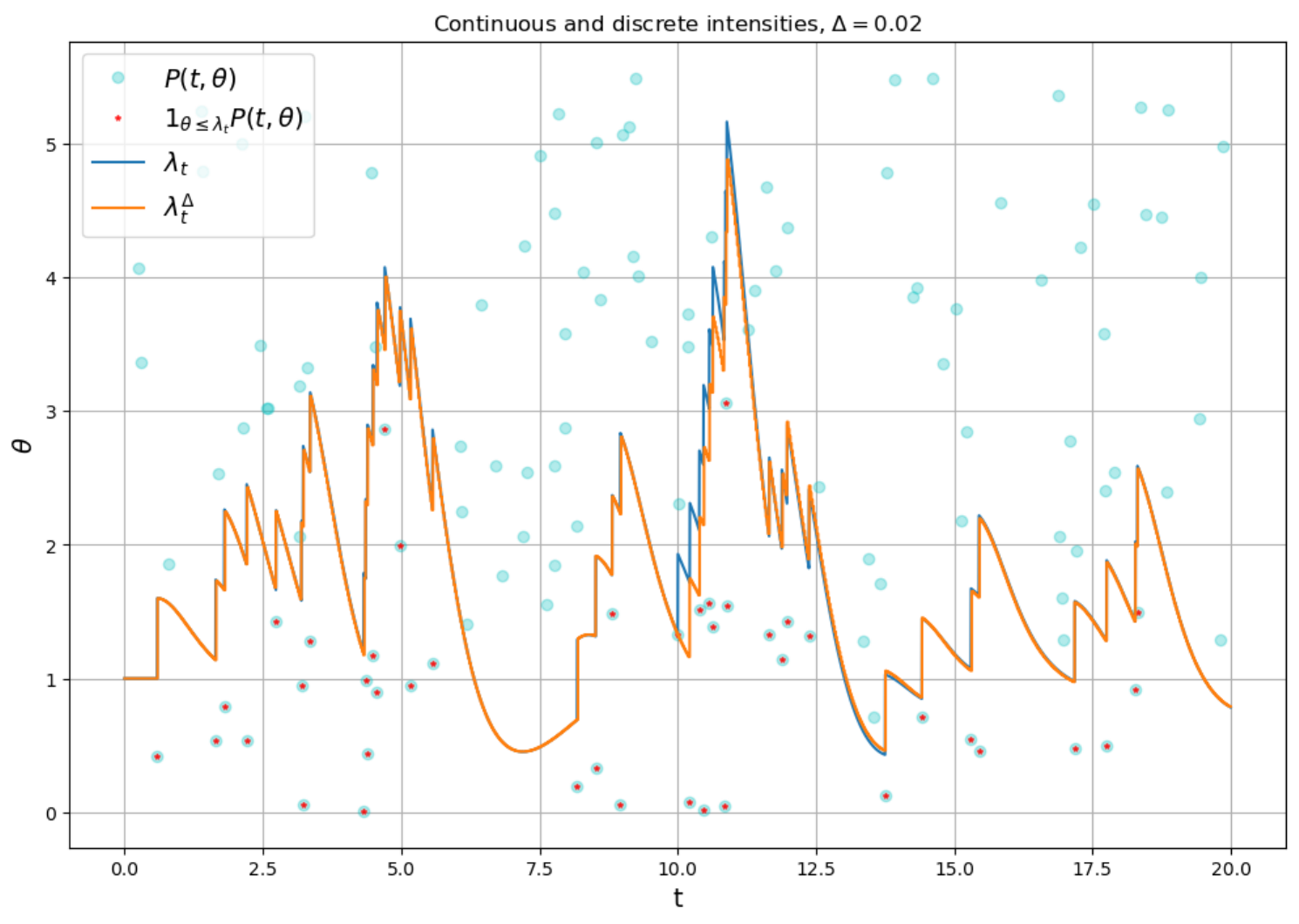}
        \caption{As the discretisation step $\Delta$ becomes smaller, the two trajectories become closer and tend to accept the exact same points.}
    \end{subfigure}
    \caption{A realisation of the discrete and continuous time intensities as thinning from the same underlying Poisson measure $P$. The jump rate is $\psi(x)=(x)_+$ and the kernel function is $h(t)=\frac{0.6\cdot \cos(t)}{1+t^2}$.}
\end{figure*}

\subsection{Preliminary results}
First, we give a regularity result on the limit process that is, the continuous-time process. This is motivated by the fact that if $\lambda$ is too irregular, a piecewise approximation of it would accept (resp. miss) many points that are not accepted (resp. rejected) in the continuous process.

\begin{Lemma}
\label{lmm:delta}
    Let $h, \psi $ and $b$ three functions satisfying Assumption \ref{ass:stability}. There exists a constant $K$ such that for any $v\in [0,T]$  we have 
    $$\E |\lambda_v-\lambda_{(v)_\Delta}|\leq K \left(  \int_0^\Delta|h(y)|\dx y + \sup_{\epsilon \in[0, \Delta]}\int_0^{T-\Delta}\left|h(y+\epsilon)-h(y)\right|\dx y\right)$$

\end{Lemma}

\begin{proof}
For $v=T,$ $\lambda_v-\lambda_{(v)_\Delta}=0$ thus we choose $v\in [0,T[.$
Since $\psi$ is $L-$Lipschitz and using a linear change of variables we have that 
\begin{align*}
    \E |\lambda_v-\lambda_{(v)_\Delta}|&\leq L \E\left | \int_0^v h(v-s)\dx \xi_s - \int_0^{(v)_\Delta}h((v)_\Delta-s)\dx \xi_s\right|\\
    &\leq L\left( \E\left |\int_{(v)_\Delta}^v h(v-s) \dx \xi_s \right|+ \E\left| \int_0^{(v)_\Delta}h(v-s)-h((v)_\Delta-s) \dx \xi_s\right|\right)\\
      &\leq L\left( \int_{(v)_\Delta}^v \left | h(v-s)\right| \E [b(Y)]\E \lambda_s\dx s +  \int_0^{(v)_\Delta}\left|h(v-s)-h((v)_\Delta-s)\right| \E [b(Y)]\E\lambda_s\dx s\right)\\
\end{align*}
      because $\dx\E [ \xi_s]= \E b(Y) \E \lambda_s \dx s$. Using the bound on the expected value proved in Lemma \ref{lmm:intensite_bornee}, we have that 
     $$\E |\lambda_v-\lambda_{(v)_\Delta}| \leq  \frac{\E [b(Y)]L\psi(0)}{1-\E [b(Y)]L\|h\|_1}\left(  \int_0^\Delta|h(y)|\dx y + \int_0^{(v)_\Delta}\left|h(y+v-(v)_\Delta)-h(y)\right|\dx y\right),$$
and the result follows immediately.
\end{proof}
This lemma shows that, the kernel's regularity in the sense of the shift operator in the $L_1$ norm, yields the intensity's regularity.\\

Since the risk processes are the result of accepting the points of the underlying Poisson measure under the intensity's curve, the approximation of the intensity is an important step in proving the convergence of $R^\Delta$ to $R$. We now give the first bound on the distance between intensities on the points of the discretisation grid.

The following constants will be useful for the sequel.
\begin{Definition}
    \label{rem-dec-const*}
{We recall that $\rho_h=L\|h\|_1\E b (Y)$ and $\rho_{h,\Delta}=L\sum_{k=1}^M |h_k|\Delta \E b(Y)$. We set} 
    \begin{align*}
        C_S(h,\Delta) &= \frac{1}{(1-\rho_h)}+\frac{1}{(1-\rho_{h,\Delta})},\\
        C_R(h,\Delta)&= \int_0^\Delta|h(y)|\dx y+ \sup_{\epsilon \in [0,\Delta]}\int_0^{T-\Delta} |h(y+\epsilon)-h(y)|\dx y+\int_0^{T-\Delta} \left | h(y)-h\left( (y)_\Delta+\Delta\right) \right| \dx y.
    \end{align*}
    The constant $C_S(h,\Delta)$ is related to the stability assumptions \ref{ass:stability} and \ref{ass:stability_discrete} while the constant $C_R(h,\Delta)$ depends on the regularity of the kernel.
\end{Definition}
%\textcolor{magenta}{est-ce qu'on remonte l'expresssion de $C_R$ quand $h$ est a $p$ variation finie ici ?}
%The following constants will be usefull for the sequel.
%\begin{Remark}

    %\label{rem-dec-const*}

    %\begin{align*}
     %   C_S(h,\Delta) &= \frac{1}{(1-\rho_h)^2}+\frac{1}{(1-\rho_{h,\Delta})^2},\\
      %  C_R(h,\Delta)&= \int_0^\Delta|h(y)|\dx y+ \sup_{\epsilon \in [0,\Delta]}\int_0^{T-\Delta} |h(y+\epsilon)-h(y)|%%\dx y+\int_0^{T-\Delta} \left | h(y)-h\left( (y)_\Delta+\Delta\right) \right| \dx y.
  %  \end{align*}
   % The constant $C_S(h,\Delta)$ is related to the stability assumptions \ref{ass:stability} and \ref{ass:stability_discrete} while the constant $C_R(h,\Delta)$ depends on the regularity of the kernel.
%\end{Remark}

\begin{Lemma}\label{lmm:comp-intens}
Let $\lambda$ (resp. $\lambda^\Delta$) be the intensities defined by thinning from the Poisson measure $P$ in Definition \ref{def:hawkes_continuous} (resp. \ref{def:hawkes_discrete}).  
Assume that Assumptions \ref{ass:stability} and \ref{ass:stability_discrete} hold  
%\textcolor{blue}{Je mets la suite en commentaire car c'est dans les hypotheses }%and let $\rho_h:=L \E[b(Y)]\|h\|<1, $ $\rho_{h,\Delta}:=\Delta L \E[b(Y)]\sum_{k=1}^{M-1} h_k<1.$ 
%Let us define
%\begin{align*}
  %   C_R(h, \Delta):=&
   % \int_0^\Delta|h(y)|\dx y   
    %+sup_{\epsilon \in[0, \Delta]}\int_0^{T-\Delta}\left|h(y+\epsilon)-h(y)\right|\dx y+
    % \int_0^{T-\Delta} \left | h(y)-h\left( (y)_\Delta+\Delta\right) \right| \dx y
     .
 %\end{align*}
There exists a constant $K$ such that 
for all  $u\in [0,T]$ we have that 
 \begin{equation}
 \label{ineq:lambda_discret-R}
     \E \left| \lambda_{(u)_\Delta}- \lambda^\Delta_u \right| \leq K {C_R(h,\Delta)C_S(h,\Delta)}.
 \end{equation}
\end{Lemma}
The proof of this Lemma is given in Section \ref{sec:appendix}.\\
%\textcolor{blue}{Je mets ce resultat ce forme de lemme pour que l'on puisse y faire reférence par la suite}
Before proving the convergence of the discrete-time Hawkes risk, we point out to the fact that if $h$ satisfies Assumption \ref{ass:stability} and is sufficiently regular, that is satisfying the following assumption

  \begin{Assumption}\label{ass:limit}

    \begin{align*} 
    \lim_{\Delta \to 0} \int_0^{T-\Delta} |h(t)-h((t)_\Delta+\Delta)| \dx t =0,
    \end{align*}
    
\end{Assumption}
then Assumption \ref{ass:stability_discrete} becomes superfluous. 
\begin{Lemma}
    \label{lmm:superfluous}
    If Assumptions \ref{ass:stability} and \ref{ass:limit} are satisfied by a kernel $h$, then for any $\epsilon >0$ small enough, there exists a threshold $\Delta_1 >0$ such that 
    $$\frac{1}{1-\rho_{h,\Delta}} \leq \frac{1}{1-\rho_h} +\epsilon$$
    for any $\Delta \leq \Delta_1$. In particular, there exists $\Delta_0 >0$ 
   Assumption \ref{ass:stability_discrete} is verified by $h$ for any $\Delta\leq \Delta_0$.
\end{Lemma}
\begin{proof}

    For a given kernel $h \in L_1$ we have that 
    \begin{align}
    \Delta\sum_{i=1}^{M-1} |h(i\Delta)| &\leq \sum_{i=1}^{M-1} \Delta\left|h(i\Delta) -\Delta^{-1} \int_{(i-1)\Delta }^{i\Delta} h(s) \dx s \right| + \sum_{i=1}^{M-1} \left| \int_{(i-1)\Delta }^{i\Delta} h(s) \dx s \right| \nonumber\\
    &\leq \sum_{i=1}^{M-1} \int_{(i-1)\Delta}^{i\Delta}| h( i\Delta) -h(s)| \dx s + \int_0^T|h(s)| \dx s \nonumber \\
    &= \int_0^{T-\Delta} | h(s) - h( (s)_{\Delta} +\Delta)| \dx s + \|h\|_1 \label {ineq: hyp_superflue}.
    \end{align}
    Since the first time in the right hand side of the last inequality tends to zero, we conclude that Assumption \ref{ass:stability_discrete} is satisfied below a certain threshold $\Delta_1 >0$.
\end{proof}
  %   \textcolor{blue}{Clairement a supprimer, }Therefore, Assumption \ref{ass:stability_discrete} becomes superfluous for functions with enough regularity, which will be defined with more detail in the next sections. 
  %   \textcolor{magenta}{It will be replaced by 
  %   \begin{Assumption}\label{hyp-h-delta}
  %       We have that 
  %       \begin{align*}
  %         \left(  \int_0^{T-\Delta} | h(s) - h( (s)_{\Delta} +\Delta)| \dx s \right) L {\mathbb E}[b(Y)] \leq \frac{1- \rho_h}{2}
  %       \end{align*}
  %   \end{Assumption}
  %   Under Assumption \ref{hyp-h-delta}, $C_S(h,\Delta)$ is bounded by a constant independent of $\Delta$
  % $$  C_S(h,\Delta) \leq \frac{3}{1- \rho_h}. $$  }
  
With these results stated, we are now ready to give the first approximation result of the risk process. We will skip the dependence in $C_S(h,\Delta)$ in forthcomming convergence results.

\subsection{Convergence in the fractional Sobolev space {and Riemann-Liouville fractional space}}
The goal of this section is to give a bound on the distance between the risk process on the compact interval $[0,T]$ and its discrete counterpart in the fractional Sobolev space
{and Riemann-Liouville fractional space (see \cite{SKM} or \cite{BLNT} for some details)}. For a measurable function $u:[0,T]\to \R$ and $\eta\in(0,1)$, we introduce the norm 
$$\|u\|_{W^{\eta,q}_T}^q:=\int_0^T |u(t)|^q \dx t + \int_0^T \int_0^T \frac{|u(t)-u(s)|^q}{|t-s|^{1+p\eta }} \dx t \dx s.$$
Naturally, the Sobolev fractional space is defined as $W^{\eta,q}_T:= \left\{ u \in L_q ([0,T]): \|u\|_{W^{\eta,q}_T} <+\infty\right \}.$
Here we restrict ourselves to $q=1$ and omit it from the notation.
We also introduce the fractional integral. For $u\in L^1([0,T])$ and $0<\eta<1$
\begin{align*}
    I_{0^+}^{\eta} (u)(t)= \frac{1}{\Gamma(\eta)}\int_0^t (t-s)^{\eta- 1} u(s) ds,~~t\in [0,T]
\end{align*}
Let us denote 
\begin{align*}
    I_{0^+}^{\eta}(L^1([0,T]))=\{u \in L^1([0,T]) \left| \exists v \in L^1([0,T]),~~u=I_{0^+}^{\eta}(v) \right.\},
\end{align*}
{
With these notations in hand, the Riemann-Liouville Fractional Sobolev space is
\begin{align*}
    W_{RL,0^+}^{\eta,1}(T)= \{u \in L^1([0,T])~~s.t. I_{0^+}^{\eta}(u) ~~\mbox{is absolutely continuous}\}.
\end{align*}}

Bergougnoux \textit{et al.} have proved the following theorem (see Theorem 3.2 of \cite{BLNT})
\begin{Theorem}
    For $  0<\eta' <\eta<1$
    \begin{align*}
        W^{1,\eta}(T)\cap I_{0^+}^{\eta'}(L^1([0,T])) \subset W_{RL,0^+}^{\eta',1}(T)
        \end{align*}
        with continuous injection.
        More precisely, there exists a constant $C$ such that  for $u = I_{0^+}^{\eta'}(v)\in W^{\eta,1},$
        \begin{align*}
            \|v\|_1\leq C \|u\|_{W^{\eta,q}_T}.
        \end{align*}
\end{Theorem}

\begin{Proposition}
Let $T>0.$ Let $( h, \psi, \nu,b)$ fulfilling Assumption \ref{ass:stability}.
    Let $R$ (resp. $R^\Delta$) be the continuous time (discrete time) risk process. 
     We have that 
    $$(R_t)_{t\in [0,T]} \in W^\eta_T \cap W_{RL,0^+}^{\eta,1}([0,T])\text { and }(R^{\Delta}_t)_{t\in [0,T]} \in W^\eta_T\cap W_{RL,0^+}^{\eta,1}([0,T]) $$
    almost surely. 
\end{Proposition}
\begin{proof}
The processes   $R$ and $R^{\Delta}$ are linear combination of indicators of finite sub-interval of the form $[a,T]$ with $a>0$ of $[0,T]$ which belongs to $W^{\eta,p}_T\cap W_{RL,0^+}^{\eta,1}([0,T])$ according to Lemma \ref{lem-indic-sob}.
\end{proof}

We now give a bound on the difference between the aggregate risk observed for the continuous-time risk process and its discrete-time counterpart, which will be crucial in proving convergence in the fractional Sobolev space. \\

\begin{Proposition}
\label{prop:main-R} Let $T>\Delta>0.$ Let $(h, \psi,\nu,b)$ fulfilling Assumptions \ref{ass:stability} {and \ref{ass:stability_discrete}}. %  \ref{hyp-h-delta}.
    Let $R$ (resp. $R^\Delta$) be a continuous time (resp. discrete time) Hawkes risk process, defined by thinning from the {same} underlying Poisson measure $P$.\\
    There exists a constant $K$ such that for all $0\leq s \leq t \leq T$ 
    \begin{align*}
        \E \left [|(R_t-R_s)-(R^\Delta_t-R^\Delta_s)|\right] \leq K
     \left(C_R(h,\Delta) \left[ (t)_\Delta-(s)_\Delta\right]
        + \Delta \right),
    \end{align*}
    where $C_R(h,\Delta) $  is given in Definition \ref{rem-dec-const*}.

    % \begin{align*}
    %     C(\Delta)= \frac{2L\psi(\mu)}{(1-L\|h\|_1)^2} &\Bigg ( \int_0^\Delta |h(y)| \dx y 
    %     + \sup_{\epsilon \in[0, \Delta]}\int_0^{T-\Delta}\left|h(y+\epsilon)-h(y)\right|\dx y\\
    %     &+\frac{1}{1-L\Delta\sum_{j=1}^{M}h_j}\int_0^{T-\Delta} \left | h(y)-h\left( (y)_\Delta + \Delta \right) \right| \dx y\Bigg)\\
    % \end{align*}

  %  \begin{align*} C'(\Delta)=&\frac{\E [b(Y)]L\psi(0)}{(1-\rho_h)^2} \Big( 
  %   (1+\E [b(Y)]L\|h\|_1)\int_0^\Delta|h(y)|\dx y   
  %  +\E [b(Y)]L\|h\|_1\sup_{\epsilon \in[0, \Delta]}\int_0^{T-\Delta}\left|h(y+\epsilon)-h(y)\right|\dx y
  %   \Big)\\
  %    &+  \frac{\E[b(Y)]L\psi(0)}{(1-\rho_{h,\Delta})(1-\rho_h)}
  %   \int_0^{T-\Delta} \left | h(y)-h\left( (y)_\Delta+\Delta\right) \right| \dx y\\
   %  &+\frac{\E b(Y)L\psi(0)}{1-\rho_h}\left(  \int_0^\Delta|h(y)|\dx y + \sup_{\epsilon \in[0, \Delta]}\int_0^{T-\Delta}\left|h(y+%\epsilon)-h(y)\right|\dx y\right)\\
   %  =&\frac{2\E b(Y)L\psi(0)}{(1-\rho_h)^2}\int_0^\Delta|h(y)|\dx y + \frac{\E b(Y)L\psi(0)}{(1-\rho_h^2}  \int_0^\Delta|h(y)|\dx y\\
   %  &+\frac{\E[b(Y)]L\psi(0)}{(1-\rho_{h,\Delta})(1-\rho_h)}
   %  \int_0^{T-\Delta} \left | h(y)-h\left( (y)_\Delta+\Delta\right) \right| \dx y
   %  \end{align*}
  
\end{Proposition}

The proof of this proposition is given in Section \ref{sec:appendix}. Before we move to strong convergence results, we give the following numerical illustration of Proposition \ref{prop:main-R} for the simple Hawkes process of kernel $h(t)=\frac{0.6\cdot \cos(t)}{1+t^2}$.\\
\begin{figure}[h!]
    \centering
    \includegraphics[width=120mm]{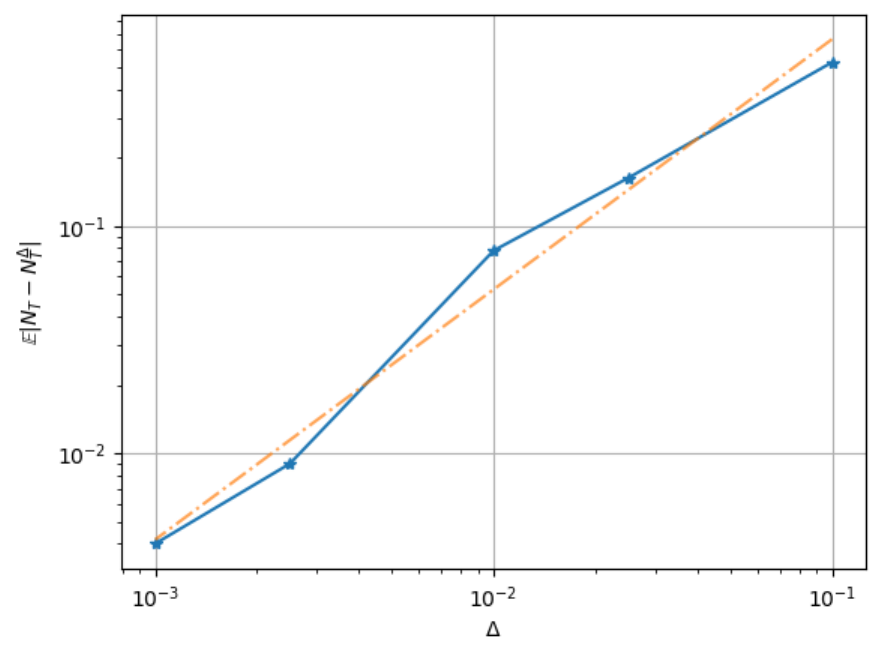}
    \caption{Blue: A Monte Carlo approximation of $\E |N_T-N^{\Delta}_T|$ for $T=5$. Orange: The least square linear approximation. Its equation is $y= 8.4 \cdot \Delta^{1.1}$. }
    \label{fig:enter-label}
\end{figure}

We are now ready to give the first general strong convergence result for the discrete-time Hawkes risk processes.
\begin{Theorem}\label{thm-Sobolev}
Let $T>0.$ Let $(h, \psi,\nu,b)$ fulfilling Assumptions \ref{ass:stability} and \ref{ass:stability_discrete}. 
    Let $R$ (resp. $R^\Delta$) be a continuous time (resp. discrete time) Hawkes risk process. We have that

    $$\E \|R-R^\Delta\|_{W^\eta_T} \leq K\left[T^{2} C_R(h,\Delta) + T\Delta^{1-\eta}\right].$$
    where 
    $C_R(h,\Delta)$ is  defined in Definition \ref{rem-dec-const*}
    and $K$ is a positive multiplicative constant depending on $\eta$ that does not depend on $T$ nor $\Delta$.
 
\end{Theorem}
{\begin{Remark}
    The rate of convergence in $\Delta$  decreases with the parameter $\eta$ and vanishes when $\eta$ goes to 1.
\end{Remark}}
The strong approximation result in Theorem \ref{thm-Sobolev} is not useful for numerical approximations, because the dependence of $C_R(h,\Delta)$ on $\Delta$ is not explicit a priori. It turns out that, if the kernel is of finite $p-$variations, one can give the order of convergence in $\Delta$. \\
For $p\geq 1$, the $p-$variation of a function $f:[0,T]\to \R$ is defined as
$$\|f\|_{p-var}:=\left(\sup _{ \mathcal D} \sum_{t_i\in\mathcal D} |f(t_{i+1})-f(t_i)|^p\right)^{1/p},$$
where $\mathcal D$ is the set of subdivisions of $[0,T]$. \\
The set of functions of finite $p-$variation contains piecewise Hölder functions of a Hölder index $p^{-1} \in(0,1]$. As we will see in the next Corollary, these functions are regular enough to have more explicit rates of convergence. Using Lemma \ref{lmn-p-var-majo}, we derive: 
\begin{Corollary}

    \label{cor:sobolev-p-var}
    Let $T> 0$ and $p\geq 1$ and assume that Assumption \ref{ass:stability} holds. Moreover, assume that the kernel $h$ is of a finite $p-$variation $\|h\|_{p-var}$. Then, for $\Delta$ small enough, {Assumption \ref{ass:limit}} is fulfilled and
    $$C_R(h,\Delta) \leq K\left[ \|h\|_{p-var} T^{\frac{p-1}{p}}\Delta^{\frac{1}{p}}+ \Delta \|h\|_{\infty}\right].$$
    Moreover
    $$\E \|R-R^\Delta\|_{W^\eta_T}\leq K_{\eta}\left[\|h\|_{p-var}T^{\frac{3p-1}{p}}\Delta^{\frac{1}{p}}+ T \Delta^{1-\eta}+\Delta \|h\|_{\infty}\right],$$
    for some positive constant $K_{\eta}$ that does not depend on $T$ nor $\Delta$. 
\end{Corollary}
{\begin{Remark}
The parameter $p$ measures the  regularity of the kernel $h.$
    The rate of convergence in $\Delta$  is a decreasing  function of $p,$ and in $T$ an increasing function of $p.$ 
\end{Remark}}
We now turn to the proof of Theorem \ref{thm-Sobolev}.
\begin{proof} 
First, we recall the result of Proposition \ref{prop:main-R}
%\begin{align*}
  %  \E|(R_t-R_s)-(R^\Delta_t-R^\Delta_s)| \leq &K \left((t)_\Delta-(s)_\Delta\right)\bigg( \int_0^\Delta |h|(r) \dx r+ \sup_{\epsilon \in [0,\Delta]}\int_0^T |h(y+\epsilon)-h(y)|\dx y \\ &+ \int_0^T |h(y)-h((y)_\Delta+\Delta)|\dx y \bigg)+ K \Delta.
%\end{align*}
\begin{align*}
    \E|(R_t-R_s)-(R^\Delta_t-R^\Delta_s)| \leq &K  \left[C_R(h,\Delta)\left((t)_\Delta-(s)_\Delta\right)+ \Delta\right].
\end{align*}

 The norm of the difference between the continuous time risk process and its discrete time counterpart in the fractional Sobolev space writes
    \begin{align*}
        \E \|R-R^\Delta\|_{W^\eta_T} &=\int_0^T \E|R_t-R^\Delta_t| \dx t +   \int_0^T \int_0^T \frac{\E|(R_t-R_s)-(R^\Delta_t-R^\Delta_s)|}{|t-s|^{1+\eta }} \dx t \dx s.
    \end{align*}
      We bound each of the three terms individually. 
    We start with second term. Thanks to the symmetry of the integrals with respect to $s$ and $t$
 
    \begin{align*}
        \int_0^T \int_0^T \frac{\E|(R_t-R_s)-(R^\Delta_t-R^\Delta_s)|}{|t-s|^{1+\eta }} &\dx t \dx s \\
        =& \sum_{i,j=1}^M \int_{(i-1)\Delta}^{i\Delta\wedge T}\int_{(j-1)\Delta}^{j\Delta\wedge T}  \frac{\E|(R_t-R_s)-(R^\Delta_t-R^\Delta_s)|}{|t-s|^{1+\eta }} \dx t \dx s\\
        =&\sum_{i=1}^M\int_{(i-1)\Delta}^{i\Delta\wedge T} \int_{(i-1)\Delta}^{i\Delta\wedge T} \frac{\E|(R_t-R_s)-(R^\Delta_t-R^\Delta_s)|}{|t-s|^{1+\eta }} \dx t \dx s\\
        &+2\sum_{i=1}^{M-1}\int_{(i-1)\Delta}^{i\Delta} \int_{i\Delta}^{(i+1)\Delta \wedge T} \frac{\E|(R_t-R_s)-(
R^\Delta_t-R^\Delta_s)|}{|t-s|^{1+\eta }} \dx t \dx s\\
        &+2\sum_{i=1}^{M-2}\sum_{j=i+2}^M  \int_{(i-1)\Delta}^{i\Delta} \int_{(j-1)\Delta}^{j\Delta\wedge T} \frac{\E|(R_t-R_s)-(R^\Delta_t-R^\Delta_s)|}{|t-s|^{1+\eta }} \dx t \dx s.
    \end{align*}
  For the first term, since $s$ and $t$ are in the same interval $[(i-1)\Delta, i \Delta),$ we have that $R_t^\Delta- R_s^\Delta=0.$ Moreover
  \begin{align*}
      \E|R_t-R_s| \leq \E \left[\sum_{i=N_s+1}^{N_t} |Y_i|\right]
      = \E[|Y|]\E[N_t-N_s]
      = K \int_s^t \E[\lambda_u]du.
      \end{align*}
      Using  Lemma \ref{lmm:intensite_bornee} and the fact that $\rho_{h} <1$ we have 
  \begin{align*}
      \E|R_t-R_s|    &\leq \frac{K(t-s)}{1-\rho_{h}}\\
      &\leq K  (t-s).
  \end{align*}
  Therefore
    \begin{align*}
         \sum_{i=1}^M \int_{(i-1)\Delta}^{i\Delta\wedge T}\int_{(i-1)\Delta}^{i\Delta\wedge T}  \frac{\E|(R_t-R_s)-(R^\Delta_t-R^\Delta_s)|}{|t-s|^{1+\eta }} \dx t \dx s =&  \sum_{i=1}^M \int_{(i-1)\Delta}^{i\Delta\wedge T}\int_{(i-1)\Delta}^{i\Delta\wedge T}  \frac{\E|R_t-R_s|}{|t-s|^{1+\eta }} \dx t \dx s\\
         \leq & K\sum_{i=1}^M \int_{(i-1)\Delta}^{i\Delta\wedge T}\int_{(i-1)\Delta}^{i\Delta\wedge T}  \frac{|t-s|}{|t-s|^{1+\eta }} \dx t \dx s\\ 
         \leq & K\sum_{i=1}^M \int_{(i-1)\Delta\wedge T}^{i\Delta\wedge T}\int_{(i-1)\Delta}^{t}  \frac{ \dx s}{(t-s)^{\eta }} \dx t \\ 
         \leq & K\sum_{i=1}^M \int_{(i-1)\Delta}^{i\Delta\wedge T} \left(t-(i-1)\Delta \right)^{1-\eta} \dx t \\
         \leq & K_{\eta} T \Delta^{1-\eta}.
    \end{align*}
    When $s$ and $t$ are in adjacent bins, $R^\Delta_t-R^\Delta_s=\sum_{k=1}^{D^\Delta_{n_t}}Y_j$ in distribution, where \\ $D^\Delta_{n_t}= \int_{((n_t-1)\Delta, n_t \Delta]\times \R_+ \times \R} \mathds 1_{\theta \leq \lambda_u^{\Delta}} P (\dx u, \dx \theta, \dx y).$ 
   Thus 
    \begin{align*}
        \sum_{i=1}^{M-1}\int_{(i-1)\Delta}^{i\Delta} \int_{i\Delta}^{(i+1)\Delta\wedge T} &\frac{\E|(R_t-R_s)-(R^\Delta_t-R^\Delta_s)|}{|t-s|^{1+\eta }} \dx t \dx s\\
        = & \sum_{i=1}^{M-1}\int_{(i-1)\Delta}^{i\Delta} \int_{i\Delta}^{(i+1)\Delta\wedge T} \frac{\E|R_t-R_s-\sum_{j=1}^{D^\Delta_{n_t}-1}Y_j|}{|t-s|^{1+\eta }} \dx t \dx s\\
        \leq & \sum_{i=1}^{M-1}\int_{(i-1)\Delta}^{i\Delta} \int_{i\Delta}^{(i+1)\Delta\wedge T} \frac{\E|R_t-R_s| + \E |Y|\E D^\Delta_{n_t}}{|t-s|^{1+\eta }} \dx t \dx s.\\
        \end{align*}
And since $\E D^\Delta_{n_t} = \int_{(n_t-1)\Delta}^{n_t \Delta}\E \lambda_u^{\Delta} \dx u ,$ we bound $\lambda_u^{\Delta} $ using Lemma \ref{lmm:rec} and the definition of $C_S(h,\Delta)$ given in Definition \ref{rem-dec-const*}

        \begin{align*}
        \sum_{i=1}^{M-1}\int_{(i-1)\Delta}^{i\Delta} \int_{i\Delta}^{(i+1)\Delta\wedge T} &\frac{\E|(R_t-R_s)-(R^\Delta_t-R^\Delta_s)|}{|t-s|^{1+\eta }} \dx t \dx s\\
        \leq & K\sum_{i=1}^{M-1}\int_{(i-1)\Delta}^{i\Delta} \int_{i\Delta}^{(i+1)\Delta\wedge T} \frac{1}{(t-s)^\eta}+ \frac{\Delta}{(t-s)^{1+\eta}} \dx t \dx s \\
        \leq & K\sum_{i=1}^{M-1}\int_{(i-1)\Delta}^{i\Delta} \Big(((i+1)\Delta-s)^{1-\eta}-(i\Delta-s)^{1-\eta}\\
        &- \Delta((i+1)\Delta-s)^{-\eta}+\Delta (i\Delta-s)^{-\eta} \Big)\dx s \\
        =& K \sum_{i=1}^{M-1}\int_{0}^{\Delta} (x+\Delta)^{1-\eta}- x^{1-\eta}-\Delta (x+\Delta)^{-\eta}+ \Delta x^{-\eta} \dx x \\
        \leq&K \sum_{i=1}^{M-1} \Delta^{2-\eta}\\
        \leq & K_{\eta}T \Delta^{1-\eta}.
    \end{align*}

  We now treat the third case scenario, when $s$ and $t$ are separated by more than one bin. Using the upper bound of Proposition \ref{prop:main-R} and keeping in mind that {  $\Delta <t-s$ thus } $((t)_\Delta-(s)_\Delta) \leq (t-s+\Delta)\leq 2(t-s)$  Thus using the result of Proposition \ref{prop:main-R} using Remark \ref{rem-dec-const*} we have that 
    \begin{align*}
        \sum_{i=1}^{M-2}\sum_{j=i+2}^M  \int_{(i-1)\Delta}^{i\Delta} \int_{(j-1)\Delta}^{j\Delta\wedge T} &\frac{\E|(R_t-R_s)-(R^\Delta_t-R^\Delta_s)|}{\E |Y_1||t-s|^{1+\eta }} \dx t \dx s \\& \leq K \sum_{i=1}^{M-2}\sum_{j=i+2}^M  \int_{(i-1)\Delta}^{i\Delta} \int_{(j-1)\Delta}^{j\Delta\wedge T} \frac{ C_R(h,\Delta)(t-s)+ \Delta }{|t-s|^{1+\eta }} \dx t \dx s\\
        &\leq  K \sum_{i=1}^{M-2} \int_{(i-1)\Delta}^{i\Delta}  \int_{(i+1)\Delta}^T \frac{C_R(h,\Delta) (t-s)+ \Delta }{|t-s|^{1+\eta }} \dx t \dx s\\
        &\leq K\sum_{i=1}^{M-2} \int_{(i-1)\Delta}^{i\Delta} C_R(h,\Delta) (t-s)^{1-\eta}+ \Delta ( (i+1)\Delta -s)^{-\eta} \dx s\\
        &\leq K_{\eta} \left[C_R(h,\Delta)T^{2-\eta} +  T\Delta^{1-\eta}\right].
    \end{align*}
    For the first term, using the fact that $t_{\Delta} \leq t$ we have
    \begin{align*}
          \int_0^T  {\E|R_t-R^\Delta_t|}\dx t  & \leq K\int_0^T  ( C_R(h,\Delta )t_{\Delta} +\Delta)\dx t \\
          &\leq K\int_0^T  ( C_R(\Delta) t +\Delta)\dx t \\
          &\leq K \left[C_R(\Delta) T^2 + \Delta T\right].
    \end{align*}
    \end{proof}

This yields the first explicit speed of convergence of the discrete-time marked Hawkes risk for a rich family of kernels. We would like to point out, that convergence rates can also be obtained for different kernels that do not lie in the set of functions with finite $p-$variations. A particularly interesting example is the Hawkes process driven by the  kernel $$h(t)=\frac{C}{\sqrt{t}}\mathds 1_{t\in (0,T]},$$
where $C>0$ is a constant that ensures that Assumption \ref{ass:stability} is in force. \\
Since this kernel is not of finite $p-$variation, one should verify that Assumption \ref{ass:stability_discrete} holds. This is possible because the sum of the inverses of the square root can be bounded by the integral of $\frac{1}{\sqrt x}$. Once this is done, we can apply Theorem \ref{thm-Sobolev}. Using the fact that $h$ is decreasing on $(0,T]$, it is possible to bound the modulus of continuity of the shift operator in an elementary fashion: 
\begin{align*}
    \int_0^{T-\Delta} |h(y+\epsilon)-h(y)|\dx y &= \int_0^{T-\Delta} h(y)-h(y+\epsilon) \dx y \\
    &=\int_0^{T-\Delta} h(y) \dx y  - \int_0^{T-\Delta} h(y+\epsilon) \dx y \\
    &=O(\Delta^{\frac{1}{2}}).
\end{align*}
Hence, $$\E \|R^\Delta- R\|_{W^\eta_T} \leq K (T \Delta^{\frac{1}{2}}+ T \Delta^{1-\eta}).$$
This rate is slower than the rate of convergence for Hawkes risks whose kernels are of bounded variations $K(T^2 \Delta + T \Delta^{1-\eta})$, which is natural because of the singularity of the inverse square root near zero. \\
We now prove the convergence in the space of càdlàg functions equipped with the Skorokhod metric for a class of Hawkes processes. 

\subsection{Convergence in the Skorokhod space}
We call $\mathbb D([0,T],{\mathbb R})$ the space of right continuous functions, with left limit (càdlàg). The canonical metric over this space is the Skorokhod metric defined by \eqref{def:Skorokhod}. The fact that this distance allows for small uncertainties in time, unlike the uniform distance, ensures that it is well adapted for Hawkes risk processes whether in continuous or discrete time. The different properties of the Skorokhod space can be found in the book \cite{Billingsley}. Let $\Lambda$ denote the class of strictly increasing continuous mapping of $|0,T]$ into itself such that $\lambda(0)=0,$
$\lambda(T)=T.$ For $x$ and $y$ in $\mathbb D([0,T])$
\begin{align*}
    d_S(x,y)= \inf_{\lambda\in \Lambda}\|\lambda-I\|_{\infty} \wedge \|x-y\|_{\infty}
\end{align*}
where $I$ is the identity map from $[0,T].$
\\

Since the jump times of the continuous time Hawkes risk $R$ occur almost surely outside of the time grid $\sigma=(k\Delta)_{k=1,\cdots,M}$, the projection $(A_\sigma R)_{t\in [0,T]} := (R_{(t)_\Delta})_{t\in [0,T]}$ represents an intermediate process between $R$ and $R^\Delta$, hence
$$d_S(R,R^\Delta) \leq d_S(R,A_\sigma R) + d_S (A_\sigma R, R^\Delta).$$
The first term on the left hand side simply evaluates the distance between the path of $R$ and its projection on the grid, and is bounded (\textit{cf.} \cite{Billingsley} Lemma 3 page 127) by
$$d_S(R,A_\sigma R) \leq \Delta \vee  \omega'_R(\Delta),$$
where $\omega'_R(\Delta)$ is the $\Delta$ modulus of continuity for càdlàg process 

    $$\omega'_X(\Delta)=\inf_{\Delta-\text{sparse}}\max _{1\leq i \leq M} \sup_{u,v \in [t_{i-1},t_{i})}|X_u-X_v|,$$
    the infimum being taken on the set of all partitions $\{t_i\}$ of $[0,T]$ such that $\min t_i-t_{i-1} > \Delta$. Therefore, the Skorokhod distance between the continuous-time Hawkes risk and its discrete-time counterpart is controlled by

$$d_S(R, R^\Delta) \leq \Delta + \omega'_R(\Delta)+ d_S (A_\sigma R, R^\Delta).$$

It is then enough to control the regularity of $R$  (using $(\omega'_R(\Delta))$ and its distance from $R^\Delta$ on the points of the grid $(d_S (A_\sigma R, R^\Delta))$. Indeed, noticing that both $A_\sigma R$ and $R^\Delta$ are constant on $[k\Delta, (k+1)\Delta)$ for $k=1,\cdots,M$, it is immediate to see that 
\begin{align*}
    d_S(A_\sigma R , R^\Delta) & \leq \|A_\sigma R - R^\Delta\|_\infty\\
    &= \max_{k=1,\cdots,M} |R_{k\Delta}-R^\Delta_{k\Delta}|.
\end{align*}
Yielding 
\begin{equation}
\label{ineq:modulus-b}
d(R,R^\Delta) \leq \Delta +  \omega'_R(\Delta) + \max_{k=1,\cdots,M} |R_{k\Delta}-R^\Delta_{k\Delta}|.
\end{equation}
Before giving an upper bound on the distance between the two processes evaluated on the points of the grid, we remind the reader that $\mathcal F$ is the filtration associated with the common underlying Poisson measure $P$. 

\begin{Proposition}
\label{prop:martingale}
    Assume that Assumptions \ref{ass:stability} and {\ref{ass:limit}} or \ref{ass:stability_discrete} are in force. Assume also that $\nu$ has a finite second moment. Let 
    $$\Xi_k:= R_{k\Delta}-\E Y\int_0^{k\Delta} \lambda _s \dx s$$
    and 
    $$\Xi^\Delta_k :=R^{\Delta}_{k\Delta} - \E Y  \sum _{i=1}^k \lambda^\Delta_{i\Delta} \Delta.$$
    Then, $(\Xi_{k})_{k=0,\cdots, M}$ and $(\Xi^\Delta_k)_{k=0,\cdots, M}$ are $(\mathcal F_{k\Delta})_{k=1,\cdots,M}-$martingales. Moreover, 
   % $$\mathbb E \left[\max_{k=1,\cdots,M} |\Xi_{k}-\Xi^\Delta_{k}|\right ] \leq \sqrt{T C_S(h,\Delta)\left(\int_0^\Delta |h(s)| \dx s + \sup _{\epsilon \in [0,\Delta]} \int_0^{T-\Delta} |h(s+\epsilon)-h(s)| \dx s + \int_0^{T-\Delta} |h(y)-h((y)_\Delta+\Delta)|\dx y \right) }.$$
   $$\mathbb E \left[\max_{k=1,\cdots,M} |\Xi_{k}-\Xi^\Delta_{k}|\right ] \leq K \sqrt{T C_R(h,\Delta)}.$$
    Where $K$ is a positive multiplicative constant that does not depend on $T$ nor $\Delta$ and $C_R(h,\Delta)$ is defined in Definition \ref{rem-dec-const*}. 
\end{Proposition}
The proof of this Proposition can be found in Section \ref{sec:appendix}. \\
We finally give an upper bound on the Skorokhod distance between the continuous-time Hawkes risk and its discrete-time counterpart, in case the jump rate $\psi$ is bounded. 
\begin{Theorem}
    \label{thm:skorokhod}
    Suppose that Assumptions \ref{ass:stability} and \ref{ass:limit} or \ref{ass:stability_discrete} are in force. Assume furthermore that $\psi$ is bounded and  that $\nu$ has a finite second moment.There exists a positive constant $K$ that does not depend on $T$ and $\Delta$ such that 
    \begin{align*}
    \E d_S(R,R^\Delta) &\leq K \left( \Delta (1+T)(1+ \|\psi\|_{\infty}) + \sqrt{T C_R(h, \Delta) } +T C_R(h, \Delta)\right).
    \end{align*}
   {Moreover, if $h$ is of finite $p$ variations for $p\geq 1$, then Assumption \ref{ass:limit} automatically holds and 
    \begin{align*}
        \E d_S&(R,R^\Delta)\\ &\leq K \left( \Delta (1+T)(1+ \|\psi\|_{\infty}+\|h\|_{\infty}) + \sqrt{  \|h\|_{p-var}T^{\frac{2p-1}{p}}\Delta^{\frac{1}{p}}  + T\Delta \|h\|_{\infty} } + \|h\|_{p-var}T^{\frac{2p-1}{p}}\Delta^{\frac{1}{p}} \right).
    \end{align*}}
\end{Theorem}
\begin{proof}
        Taking the expected value of inequality \eqref{ineq:modulus-b} we have that 
    \begin{align*}
        \E d_S(R,R^\Delta) &\leq \Delta + \E \omega '_R (\Delta)+  \E \max_{k=1,\cdots,M} |R_{k\Delta}-R^\Delta_{k\Delta}| \\
        & \leq \Delta + \E \omega '_R (\Delta) + \E \max_{k=1,\cdots,M} |\Xi_k-\Xi^\Delta_{k}| + \E  \max_{k=1,\cdots,M} \left | \int_0^{k\Delta}\lambda_s \dx s- \sum_{i=1}^k \lambda_{i\Delta}^\Delta \Delta  \right |\\
        &= \Delta+ A_1 +A_2 +A_3.
    \end{align*}
 The difference between the values of the discrete time and continuous time processes have been separated into a martingale part and a compensator part. \\
    
    The martingale term $A_2$ has already been dealt with in Proposition \ref{prop:martingale}. When it comes to $A_3$ we simply write 
    \begin{align*}
         \left | \int_0^{k\Delta}\lambda_s \dx s- \sum_{i=1}^k \lambda_{i\Delta}^\Delta \Delta  \right |& =\left | \sum _{i=1}^k \int_{(i-1)\Delta}^{i\Delta} \lambda_ s - \lambda^\Delta_{i\Delta} \dx s\right |\\
        & \leq \sum _{i=1}^M  \int_{(i-1)\Delta}^{i\Delta}  |\lambda_s - \lambda_{i\Delta}^\Delta| \dx s\\
        &\leq K T C_R(h,\Delta)
    \end {align*}
   using Inequality \eqref{ineq:lambda_discret-R}.
\\
The last term is $A_1$. Since the jump rate is bounded by $\|\psi\|_\infty$, the  compound Poisson process 
$$\Pi_t= \int_{(0,t]\times \R_+ \times \R}|y| \mathds 1_{\theta \leq \|\psi\|_\infty} P (\dx s, \dx \theta, \dx y)$$
dominates the process $R$. That is, for any $0\leq a \leq b \leq T$, $\Pi_b-\Pi_a \geq \left|R_b-R_a\right|.$ 
%We denote the marks of the points under the line $\|\psi\|_\infty$ by $Y_1,\cdots, Y_n$. We also define $\delta(k)$ to be the indicator that the point $(t_k,\theta_k,Y_k)$ is also accepted by $N$ (that is under the height $\theta$ is below the curve of% $\lambda$). \\
%We then have that, for two points $\tau_{i-1}$ and $\tau_i$ 
%\begin{align*}
%    \sup_{u,v \in [\tau_{i-1}, \tau_i)} |R_u-R_v| &\leq \sum _{t_j \in [\tau_{i-1}, \tau_i )}|Y_j| \delta (j)\\
%    &\leq  \sum _{t_j \in [\tau_{i-1}, \tau_i )}|Y_j|\\
%    &\leq K (\Pi_{\tau_i}-\Pi_{\tau_{i-1}})
%\end{align*}
%where $K$ is a constant that bounds $Y$ from above. \\
The problem now is to determine the behaviour of the modulus of continuity of a compound Poisson process of intensity $\|\psi\|_\infty $, which is solved in Lemma \ref{lmm:module_poisson}. 
\end{proof}
{The case of a bounded jump rate is quite restrictive, for instance the results of Theorem \ref{thm:skorokhod} cannot be applied to the standard linear Hawkes process. We hence give a generalisation to unbounded jump rates in the following theorem.} 
\begin{Theorem}
    \label{thm:skorokhod-nb}
    Suppose that Assumptions \ref{ass:stability} and \ref{ass:limit} or \ref{ass:stability_discrete} are in force. Assume that $h$ is bounded and that $\nu$ has a finite second moment. There exists a positive constant $K$ that does not depend on $T$ and $\Delta$ such that 
    
    \begin{align*}
    \E d_S(R,R^\Delta) &\leq K \left( \sqrt{\Delta} (1+T^{\frac{3}{2}})+ \sqrt{T C_R(h, \Delta) } +T C_R(h, \Delta)\right).
    \end{align*}
    {Moreover, if $h$ is of finite $p$ variations for $p\geq 1$, then Assumption \ref{ass:limit} automatically holds and 
    \begin{align*}
        \E d_S(R,R^\Delta) &\leq K \left( \sqrt{\Delta} (1+T^{\frac{3}{2}})+T\Delta + \sqrt{  \|h\|_{p-var}T^{\frac{2p-1}{p}}\Delta^{\frac{1}{p}}  + T\Delta \|h\|_{\infty} } + \|h\|_{p-var}T^{\frac{2p-1}{p}}\Delta^{\frac{1}{p}} \right).
    \end{align*}}
\end{Theorem}
\begin{proof}
%First we prove that 
%\begin{align}\label{esp-sup-lambda}
 %   \E   \sup_{t\leq T}   \lambda_t \leq \psi(0) + T\frac{\psi(0)}{1- \E b(Y) L \|h\|_1}
%\end{align}
   % Recall that 
    %\begin{align*}
     %  \lambda_t &=\psi \left( \int_{[0,t) \times \R_+ \times \R} h(t-s) \mathds 1_{\theta \leq \lambda _s} b(y)P(\dx s, \dx \theta, \dx y) \right) 
     % = \psi \left( \int_0^{t-} h(t-s) d\xi_s\right).
   % \end{align*}
   % Thus for all $t \in [0,T]$
   % \begin{align*}
    %\sup_{t\leq T}   \lambda_t &\leq \psi (0) + L\sup_{t\leq T} \int_0^{t-} |h(t-s)| d\xi_s \\
    %  &\leq \psi(0) + L \|h\|_{\infty} \xi_T
   % \end{align*}
 
   % where $\xi$ is the increasing process  $\xi_t=\sum_{k=1}^{N_t}b(Y_k)=\int_{(0,t] \times \R_+ \times \R} b(y)\mathds 1_{\theta \leq \lambda _s} P(\dx s, \dx \theta, \dx y).$
%Since $P$ is a Poisson measure with intensity $\dx s \dx \theta \dx y$
%\begin{align*}
%   \E   \sup_{t\leq T}   \lambda_t  &\leq  \psi (0) + L \|h\|_\infty\E b(Y) \int_0 ^T \E \lambda _t \dx t \\
%\end{align*}
%Using Lemma \ref{lmm:intensite_bornee}, we obtain 
%\begin{align*}
  %  \E   \sup_{t\leq T}   \lambda_t \leq \psi(0) \left( 1 + T\frac{L \|h\|_\infty \E b(Y)}{1- %\rho_{h}} \right) 
%\end{align*}
%which is exactly \eqref{esp-sup-lambda}.

%Using the same kinds of computations
%\begin{align}\label{esp-sup-lambda-delta}
  %  \E   \sup_{t\leq T}   \lambda_t^{\Delta} \leq  \psi(0) \left( 1 + T\frac{L \|h\|_\infty \E b(Y)}{1- \rho_{h,\Delta}} \right). 
%\end{align}

Let $C$  a positive real number to be fixed later and 
$\psi^C= \psi \wedge C.$ We also denote by $ \lambda^C,~~N^C,~~R^C$ the intensity, Hawkes  and risk processes,  $ \lambda^{C,\Delta},~~N^{C,\Delta},~~R^{C,\Delta}$ the discrete  intensity and risk processes associated to $\psi^C.$
On the event $\Gamma_C=\{\omega, \sup_{t\leq T}   \lambda_t \leq C\},$  the process $(\lambda^C_t,N^C_t,R^C_t,~~t\in [0,T])$ and $(\lambda_t,N_t, R_t,~~t\in [0,T])$ coincide, and on the event $\Gamma_{C,\Delta}=\{\omega, \sup_{t\leq T}   \lambda_t^{\Delta} \leq C\},$  the processes $(\lambda^{C,\Delta}_t,N^{C,\Delta}_t,R^{C,\Delta}_t,~~t\in [0,T])$ and $(\lambda^{\Delta}_t,N_t^{\Delta}, R^{\Delta}_t,~~t\in [0,T])$ coincide.

On the event $\Omega \setminus \left(\Gamma_C\cap \Gamma_{C,\Delta}\right),$ we estimate
\begin{align*}
    d_S(R,R^{\Delta}) \leq \sup_{s\leq T}|R_s| + \sup_{s\leq T} |R_s^{\Delta}|
     \leq \sum_{k=1}^{N_T}|Y_k| + \sum_{k=1}^{N_T^{\Delta}}|Y_k|.
\end{align*}
Thus,
\begin{align*}
    \E d_S(R,R^{\Delta}) &= \E d_S(R,R^{\Delta})\mathds 1_{\Gamma_C \cap \Gamma_{C,\Delta}} + \E d_S(R,R^{\Delta}) (\mathds 1_{\Gamma_C^c \cup \Gamma_{C,\Delta}^c})\\
    &\leq  \E d_S(R^C,R^{C,\Delta})+ \E d_S(R,R^{\Delta}) (\mathds 1_{\Gamma_C^c } + \mathds 1_{ \Gamma_{C,\Delta}^c}). 
\end{align*}
Using Cauchy-Schwartz and Markov inequalities
\begin{align*}
    \E d_S(R,R^{\Delta}) & \leq  \E d_S( R^C,R^{C,\Delta }) + 
  \sqrt{{4} \E \left(\left[\sum_{k=1}^{N_T}|Y_k|\right]^2 +\left[\sum_{k=1}^{N_T^{\Delta}}|Y_k|\right]^2\right) \left({\mathbb P} (\Gamma_C^c ) +{\mathbb P} (\Gamma_{C,\Delta}^c )  \right)}\\ 
    &\leq \E d_S( R^C,R^{C,\Delta }) + 
  \frac{1}{C}  \sqrt{2 \E \left(\left[\sum_{k=1}^{N_T}|Y_k|\right]^2 +\left[\sum_{k=1}^{N_T^{\Delta}}|Y_k|\right]^2\right) \E \left(\sup_{s \leq T} \lambda_s^2 + \sup_{s\leq T}(\lambda_s^{\Delta})^2\right)}
    .
\end{align*}
Now, we bound each term under the square root.

First, since $N_t =\int_{(0,t] \times \R_+ \times \R} \mathds 1_{\theta \leq \lambda _s} P(\dx s, \dx \theta, \dx y)$
\begin{align*}
\E [ N_T^2] &\leq 2 \int_0^T  \E[\lambda_s ]\dx s + 2\E\left[\int_0^T \lambda_s \dx s \right]^2\\
&\leq 2T \sup_{s\leq T} \E[\lambda_s ] + 2T^2\sup_{s\leq T} \E[\lambda_s^2 ].
\end{align*}
According to lemmas  \ref{lmm:intensite_bornee} and \ref{lem-moment-2} 
\begin{align}\label{maj-moment-ordre2}
   \E [ N_T^2] \leq K(1+T^2) .
\end{align}
Second, since $N_t^{\Delta} =\int_{(0,t] \times \R_+ \times \R} \mathds 1_{\theta \leq \lambda _s^{\Delta}} P(\dx s, \dx \theta, \dx y)$
\begin{align*}
\E [ (N_T^{\Delta})^2] &\leq 2 \int_0^T  \E[\lambda_s ^{\Delta}]\dx s + 2\E\left[\int_0^T \lambda_s^{\Delta}\dx s \right]^2\\
&\leq 2T \sup_{s\leq T} \E[\lambda_s ^{\Delta}] + 2T^2\sup_{s\leq T} \E[(\lambda_s^{\Delta})^2 ].
\end{align*}
According to lemmas  \ref{lmm:rec} and \ref{lem-moment-2-discret}
\begin{align}\label{maj-moment-ordre2-delta}
   \E [ (N_T^\Delta)^2] \leq K(1+T^2) 
\end{align}
Third, recall that
\begin{align*}
    \lambda_t = \psi \left( \int_{0}^{t-} h(t-s) \dx \xi_s\right)
\end{align*}
Thus since $\psi$ is Lipschitz continuous
\begin{align*}
    \lambda_s
    &\leq \psi(0) +L \int_{0}^{t-} h(t-s) \dx \xi_s
\end{align*}
Since $\xi_t= \int_{(0,t] \times \R_+ \times \R} b(y)\mathds 1_{\theta \leq \lambda _s} P(\dx s, \dx \theta, \dx y)$ is an increasing  process and $h$ is bounded
\begin{align*}
    \lambda_s
    &\leq \psi(0) +L \|h\|_{\infty} \xi_T.
\end{align*}
Thus 
\begin{align*}
  \sup_{s\leq T}  \lambda_s^2
    &\leq  2 \psi(0) + 2L^2 \|h\|_{\infty}^2 \xi_T^2.
\end{align*}
Then, taking the expectation of each term,
\begin{align*}
     \E \left[\sup_{s \leq T} \lambda_s^2\right] \leq  2 \psi(0) + 2L^2 \|h\|_{\infty}^2 \E \left[\xi_T^2\right].
     \end{align*}
     Using the same line as the proof of  the estimation of $\E[N_T^2],$ \eqref{maj-moment-ordre2}, we have 
\begin{align}\label{maj-esp-sup-lambda}
     \E \left[\sup_{s \leq T} \lambda_s^2\right] \leq  K(1+T^2). 
     \end{align}
Fourth, recall that $\lambda_t^{\Delta}= \psi \left(\sum_{k=1}^{n_t-1}h(n_t-k) X_k^{\Delta}\right). $ Since $\psi$ is Lipschitz continuous,
     \begin{align*}
         \lambda_t^{\Delta} \leq \psi(0) + L\sum_{k=1}^{n_t-1} h(n_t-k) X_k^{\Delta}.
     \end{align*}
Using the fact that $h$ is bounded and $$X_n^{\Delta} =\int_{((n-1)\Delta,n\Delta]\times \R_+ \times \R} b(y) \mathds 1_{\theta \leq l^\Delta_n} P (\dx s , \dx \theta, \dx y ),$$
     \begin{align*}
         \sup_{s\leq T} \lambda_s^{\Delta} \leq \psi(0) + L\|h\|_{\infty}\int_{(0,T]\times \R_+ \times \R} b(y) \mathds 1_{\theta \leq \lambda^\Delta_s} P (\dx s , \dx \theta, \dx y ). 
     \end{align*}
     Using the same lines as the proof of estimation \eqref{maj-esp-sup-lambda} we obtain 
     \begin{align}\label{maj-esp-sup-lambda-d_*}
     \E \left[\sup_{s \leq T} (\lambda_s^{\Delta})^2\right] \leq  K(1+T^2).
     \end{align}
Then, for a universal constant $K$ 
\begin{align*}
    \E d_S(R,R^{\Delta}) &\leq \E d_S( R^C,R^{C,\Delta }) \\
    &+     K \frac{ 1+T^2}{C}.
\end{align*}
The quantity $\E d_S( R^C,R^{C,\Delta })$ is bounded using  Theorem \ref{thm:skorokhod} for $\psi^C$ whith $\|\psi^C\|_{\infty}\leq C$
\begin{align*}
    \E d_S(R,R^{\Delta}) &\leq  K \left( \Delta (1+T)(1+ C) + \sqrt{T C_R(h, \Delta) } +T C_R(h, \Delta)\right)\\
    &+    K \frac{ 1+T^2}{M}.
\end{align*}
Taking $C=\frac{\sqrt{T}}{\sqrt{\Delta}}$ we derive 
\begin{align*}
    \E d_S(R,R^{\Delta}) &\leq  K  \sqrt{\Delta} (1+T^{\frac{3}{2}})
    +K\left( \sqrt{T C_R(h, \Delta) } +T C_R(h, \Delta)\right)
\end{align*}
where $K$ is a constant independent of $T$ and $\Delta.$ 
\end{proof}
{\begin{Remark}
    If in the proof  of Theorem \ref{thm:skorokhod-nb}, we bound ${\mathbb P}(\Gamma_C^c)$ by
    $\frac{{\mathbb E}[ \sup_{s\leq T} \lambda_s]}{C}$ instead of  $\frac{{\mathbb E}[ \sup_{s\leq T} \lambda_s^2]}{C^2}$ we obtain
    \begin{align*}
    \E d_S (R,R^{C,\Delta}) & \leq K \left ( \Delta ^{1/3} T^{4/3} \sqrt{T C_R(h,\Delta)} + T C_R(h,\Delta)\right).
\end{align*}
The power of $T$ in the constant is smaller than in the bound given  in Theorem \ref{thm:skorokhod-nb} but the rate of convergence in $\Delta$  is smaller.
\end{Remark}}

\section{Conclusion}
\label{sec:conclusion}
Using coupling arguments based on thinning from a given Poisson measure, we derived explicit bounds on the distance between the continuous time embedding of non-linear Poisson autoregression (here referred to as the discrete time Hawkes process) and the standard continuous time Hawkes process, both in the Sobolev and the Skorokhod spaces. Our bounds yield a quantitative generalisation of the convergence result proven in \cite{mahmoud}. More specifically, the speed of convergence is given both in the time step of the discretisation $\Delta$ and the time horizon $T$.\\
An interesting development of the results shown in this paper is their extension to stochastic differential equations involving both a Brownian noise and Hawkes jumps in the Skorokhod metric. To the best of our knowledge, such results do not exist in the literature where jumps are supposed to follow a Poisson process \cite{BRUTILIBERATI2007982}. We thus consider that they would constitute a valuable contribution to the approximation of jump diffusion SDEs with potential numerical applications in many fields such as option pricing or neuro-sciences. 

\section{Proofs of the convergence results}

\label{sec:appendix}
\subsection{Proof of Lemma \ref{lmm:comp-intens}}
First, we compute the distance between the two intensities at a point of the grid.
For a given $u \in [0,T[$ we write the expressions of the intensities in order to find an upper bound of $  \E \left| \lambda_{(u)_\Delta}- \lambda^\Delta_u \right| $:
     
    \begin{align*}
        \E &\left| \lambda_{(u)_\Delta}- \lambda^\Delta_u \right|\\ =& \E \left|\psi \left(  \int_0^{(u)_\Delta}h\left((u)_\Delta-v\right) \dx \xi_v\right)-\psi \left(\sum_{k=1}^{n_u-1} h_{n_u-k}X^\Delta_k\right) \right|\\
        \leq & L \E \left|\int_{(t_{n_u-1},(u)_\Delta]\times \R_+^2}h((u)_\Delta-v)b(y)\mathds 1_{\theta \leq \lambda_{v}}P(\mathrm d v, \mathrm d \theta, \mathrm d y )\right |\\ 
        &+L\E\left|\sum_{k=1}^{n_u-1}\int_{(t_{k-1},t_k]\times \R_+^2} 
 \left(h((u)_\Delta-v)\mathds 1_{\theta \leq \lambda_{v}}-h_{n_u-k}\mathds 1_{\theta \leq \lambda^{\Delta}_{t_k}}\right) b(y)P(\mathrm d v, \mathrm d \theta, \mathrm d y ) \right|\\
 \leq & \E [b(Y)]L\int_{((u)_\Delta-\Delta,(u)_\Delta]}|h((u)_\Delta-v)|\E \lambda_v \dx v \\
 &+ L \sum_{k=1}^{n_u-1} \E \left |\int_{(t_{k-1},t_k]\times \R_+^2} 
 \left(h((u)_\Delta-v)\mathds 1_{\theta \leq \lambda_{v}}-h((u)_\Delta-v)\mathds 1_{\theta \leq \lambda^{\Delta}_{t_k}}\right) b(y)P(\mathrm d v, \mathrm d \theta, \mathrm d y ) \right |\\
 &+ L \sum_{k=1}^{n_u-1} \E \left |\int_{(t_{k-1},t_k]\times \R_+^2} 
 \left(h((u)_\Delta-v)\mathds 1_{\theta \leq \lambda^\Delta_{t_k}}-h_{n_u-k}\mathds 1_{\theta \leq \lambda^{\Delta}_{t_k}}\right) b(y)P(\mathrm d v, \mathrm d \theta, \mathrm d y ) \right |\\
  \leq & K\int_{((u)_\Delta-\Delta,(u)_\Delta]}|h((u)_\Delta-v)|\E \lambda_v \dx v \\
 &+ L \sum_{k=1}^{n_u-1} \int_{(t_{k-1},t_k]} 
 \left |h((u)_\Delta-v)\right | \E |\lambda_v -\lambda^\Delta_v|\mathrm d v \\
 &+ L \sum_{k=1}^{n_u-1} \int_{(t_{k-1},t_k]} 
 \left |h((u)_\Delta-v)-h_{n_u-k} \right | \E\lambda^{\Delta}_{v}\mathrm d v\\
    \end{align*}
    where $K$ is a constant independent of $T$ and $\Delta.$
    In order to apply the same induction used in the proof of Lemma \ref{lmm:rec}, the quantity in the integral $\int_{(0,(u)_\Delta-\Delta]}$ and the term to the left $\E \left| \lambda_{(u)_\Delta}- \lambda^\Delta_u \right|$ should coincide. That is why we take the projection of $\lambda_v$ on the discretisation grid. Therefore, using the upper bounds on $\E \lambda_v$ and $\E \lambda_v^\Delta$ (\textit{cf.} Lemmas \ref{lmm:intensite_bornee} and \ref{lmm:rec}) we have that 
    \begin{align*}
        \E &\left| \lambda_{(u)_\Delta}- \lambda^\Delta_u \right|\\ \leq & K\left( \frac{1}{1-\rho_{h}} \int_{((u)_\Delta-\Delta,(u)_\Delta]}|h((u)_\Delta-v)| \dx v \right) \\
        &+K\left ( \frac{1}{1-\rho_{h,\Delta}}\sum_{k=1}^{n_u-1} \int_{(t_{k-1},t_k]} 
 \left |h((u)_\Delta-v)-h_{n_u-k} \right | \mathrm d v\right)\\
 &+{\mathbb E}[b(Y)]L \sum_{k=1}^{n_u-1} \int_{(t_{k-1},t_k]} 
 \left |h((u)_\Delta-v)\right | \E |\lambda_{t_{k-1}}- \lambda^\Delta_{t_{k-1}}|\mathrm d v + \int_{(t_{k-1},t_k]} 
 \left |h((u)_\Delta-v)\right | \E |\lambda_{v}- \lambda_{(v)_\Delta}|\mathrm d v\\
 % \leq & L\left( \frac{\psi(\mu)}{1-L\|h\|_1} \int_0^\Delta|h(y)|\dx y  \right) \\
 %        &+L\left ( \frac{\psi(\mu)}{1-L\Delta\sum_{j=1}^{+\infty}h_j}(n_u-1)\Delta \sup_{|x-y|\leq \Delta}|h(x)-h(y)|\right)\\
 % &+L \sum_{k=1}^{n_u-1} \int_{(t_{k-1},t_k]} 
 % \left |h((u)_\Delta-v)\right | \E |\lambda_{t_{k-1}}- \lambda^\Delta_{t_{k-1}}|\mathrm d v + \int_{(t_{k-1},t_k]} 
 % \left |h((u)_\Delta-v)\right | \E |\lambda_{v}- \lambda_{(v)_\Delta}|\mathrm d v.\\
 \leq & K\left( \frac{1}{1-\rho_h} \int_0^\Delta|h(y)|\dx y  \right) \\
 &+K\left ( \frac{1}{1-\rho_{h,\Delta}}\sum_{k=1}^{n_u-1} \int_{(t_{k-1},t_k]} 
 \left |h((u)_\Delta-v)-h_{n_u-k} \right | \mathrm d v\right)\\
 &+{\mathbb E}[b(Y)]L\sum_{k=1}^{n_u-1} \E |\lambda_{t_{k-1}}- \lambda^\Delta_{t_{k-1}}|\int_{(t_{k-1},t_k]} 
 \left |h((u)_\Delta-v)\right | \mathrm d v\\
 &+  \frac{K}{1-\rho_h}\left(  \int_0^\Delta|h(y)|\dx y + \sup_{\epsilon \in[0, \Delta]}\int_0^{T-\Delta}\left|h(y+\epsilon)-h(y)\right|\dx y\right)
    \end{align*}
    where we got the fourth line from Lemma \ref{lmm:delta}, and the first line using a linear time change. \\
    For the second line, we notice that $t_{k-1}=t_k-\Delta=(v)_\Delta$, hence:
    
    \begin{align*}
        \sum_{k=1}^{n_u-1} \int_{(t_{k-1},t_k]} \left |h((u)_\Delta-v)-h_{n_u-k} \right |\mathrm d v&= \sum_{k=1}^{n_u-1} \int_{(t_{k-1},t_k]} \left |h((u)_\Delta-v)-h((u)_\Delta-k\Delta) \right |\mathrm d v\\
        &=\sum_{k=1}^{n_u-1} \int_{(t_{k-1},t_k]} \left |h((u)_\Delta-v)-h((u)_\Delta-(v)_\Delta+\Delta) \right |\mathrm d v\\
        &=\int_0^{(u)_\Delta-\Delta}\left |h((u)_\Delta-v)-h((u)_\Delta-(v)_\Delta+\Delta) \right |\mathrm d v\\
    \end{align*}
    and using the change of variable $y=(u)_\Delta-v$ and noticing that $(u)_\Delta$ is already on the discretisation grid, we have that
    \begin{align*}
        \sum_{k=1}^{n_u-1} \int_{(t_{k-1},t_k]} \left |h((u)_\Delta-v)-h_{n_u-k} \right |\mathrm d v&=\int_\Delta^{(u)_\Delta} \left | h(y)-h\left( (u)_\Delta - ((u)_\Delta-y)_\Delta + \Delta \right) \right| \dx y \\
        &=\int_\Delta^{(u)_\Delta} \left | h(y)-h\left( (y)_\Delta + \Delta \right) \right| \dx y \\
        &\leq \int_0^{T-\Delta} \left | h(y)-h\left( (y)_\Delta + \Delta \right) \right| \dx y .
    \end{align*}

%\begin{align*}
      %  \sum_{k=1}^{n_u-1} \int_{(t_{k-1},t_k]} \left |h((u)_\Delta-v)-h_{n_u-k} \right |\mathrm d v&= \sum_{k=1}^{n_u-1} \int_{(n_u-k)\Delta}^{(n_u-k-1)\Delta} \left |h((u)_\Delta-v)-\Delta^{-1}\int_{(u)_\Delta-(k-1)\Delta}^{(u)_\Delta-k\Delta} h(z) \mathrm d z\right|\mathrm d v
       % \end{align*}
       % \textcolor{magenta}{We perform the change of variable $v= n_u -k-v$}
     % \begin{align*}
       % \sum_{k=1}^{n_u-1} \int_{(t_{k-1},t_k]} \left |h((u)_\Delta-v)-h_{n_u-k} \right |\mathrm d v&  
       % &\leq \sum_{k=1}^{n_u-1} \int_{(n_u -k)\Delta}^{(n_u-k+1)\Delta}\left|h(s) - h_{n_u-k} \right |\mathrm d s
       % \end{align*}
       % \textcolor{magenta}{We perform the change of variable of summation  $i= n_u -k$}
 %\begin{align*}
      %  \sum_{k=1}^{n_u-1} \int_{(t_{k-1},t_k]} \left |h((u)_\Delta-v)-h_{n_u-k} \right |\mathrm d v&  
      %  &\leq \sum_{i=1}^{n_u-1} \int_{i\Delta}^{(i+1)\Delta}\left|h(s) - h_{i} \right |\mathrm d s
      %  \end{align*}
%\textcolor{magenta}{Using the definition of $h_{i}$ one has}
      % \begin{align*}
       %\sum_{k=1}^{n_u-1} \int_{(t_{k-1},t_k]} \left |h((u)_\Delta-v)-h_{n_u-k} \right |\mathrm d v&  
       % &\leq \sum_{i=1}^{n_u-1} \int_{i\Delta}^{(i+1)\Delta}\left| h(v) -\Delta^{-1}\int_{(v)_{\Delta}-\Delta}^{(v)_{\Delta}}\right |\mathrm d v
       % \leq \int_{\Delta}^{T}\left| h(v) -\Delta^{-1}\int_{(v)_{\Delta}-\Delta}^{(v)_{\Delta}} h(s) \dx s\right |\mathrm d v
       % \end{align*}  

% \textcolor{magenta}{ fin**************************************************************************}
    For the third line  we notice that $\int_{(t_{k-1},t_k]} 
 \left |h((u)_\Delta-v)\right | \dx v $ depends only on the difference $n_u-k$, indeed 
 \begin{align*}
     \E [b(Y)]L\int_{t_{k-1}}^{t_k} \left |h((u)_\Delta-v)\right | \dx v & =\E [b(Y)]L\int_{(n_u-k)\Delta }^{(n_u-k+1)\Delta} |h(y)| \dx y\\
     &:=\eta^\Delta_{n_u-(k-1)}.
 \end{align*}
 By denoting $n_u=n$, $g_n=\E|\lambda_{n\Delta}-\lambda^\Delta_{n\Delta}|$ and $\tilde \eta ^\Delta_j=\eta^\Delta_j \mathds 1_{j\geq 2}$, we have that 
 $$g_n\leq C(h,\Delta)+\sum_{k=1}^{n-1} g_k \tilde \eta ^\Delta_{n-k},$$
 where \begin{align*}
     C(h,\Delta)=&K \frac{1}{1-\rho_{h}} \Big( 
     \int_0^\Delta|h(y)|\dx y   
    +\sup_{\epsilon \in[0, \Delta]}\int_0^{T-\Delta}\left|h(y+\epsilon)-h(y)\right|\dx y
     \Big)\\
      &+ K \frac{1}{1-\rho_{h,\Delta}}
     \int_0^{T-\Delta} \left | h(y)-h\left( (y)_\Delta+\Delta\right) \right| \dx y
     .
 \end{align*}

{Since $\tilde \eta^\Delta \in l_1$ and its sum is bounded by $\E[b(Y)]L\|h\|_1<1,$ we apply the same induction in the proof of Lemma \ref{lmm:rec} to obtain 
 $$\E \left| \lambda_{(u)_\Delta}- \lambda^\Delta_u \right| \leq \frac{C(h,\Delta)}{1-\rho_{h}}.$$
 Using Definition \ref{rem-dec-const*} we finally conclude that 
 $$\E \left| \lambda_{(u)_\Delta}- \lambda^\Delta_u \right| \leq K{C_R(h,\Delta)C_S(h,\Delta)}.$$}
\subsection{Proof of Proposition \ref{prop:main-R}}
\begin{proof}
    We first start by projecting both $s$ and $t$ on the time grid 
    \begin{align*}
        \E \left [|(R_t-R_s)-(R^\Delta_t-R^\Delta_s)| \right]=&\E \left [|(R_t-R_{(t)_\Delta})+(R_{(s)_\Delta}-R_s)+(R_{(t)_\Delta}-R_{(s)_\Delta})-(R^\Delta_t-R^\Delta_s)| \right]\\
        \leq& \E\left[|R_t-R_{(t)_\Delta}|\right]+ \E\left[|R_s-R_{(s)_\Delta}|\right]\\
        %&+\E\left[ |(R_{(t)_\Delta}-R_{(s)_\Delta})-(R^\Delta_t-R^\Delta_s)|\right]\\
       % \leq& \E\left[|R_t-R_{(t)_\Delta}|\right]+ \E\left[|R_s-R_{(s)_\Delta}|\right]\\
        &+\E\left[ |(R_{(t)_\Delta}-R_{(s)_\Delta})-(R^\Delta_{(t)_\Delta}-R^\Delta_{(s)_\Delta})|\right].
    \end{align*}
    
We handle the first two terms as follows:
\begin{align*}
     \E\left[ | R_{t} - R_{(t)_{\Delta}} |\right] \leq & \E 
     \left[ \left | \int_{((t)_{\Delta},t] \times \R_+ \times \R} 
    \mathds 1_{\theta \leq \lambda_{u}} y P(\dx u,\dx \theta, \dx y) \right|\right] \\
     \leq& \int_{\R_+}|y| \nu(\dx y) \E \left[ \int_{((t)_{\Delta},t]}\lambda_{u} \dx u \right] \\
     \leq& {\E |Y|}\int_{((t)_{\Delta},t]} \E[\lambda_{u}] \dx u \\
     \leq& \frac{K\Delta}{1-\rho_{h}}.\\
\end{align*}

    Hence
    \begin{equation}
        \label{ineq:facile-R}
        \E\left[|R_t-R_{(t)_\Delta}|\right]+ \E\left[|R_s-R_{(s)_\Delta}|\right]\leq K\frac{\Delta}{1-\rho_{h}}.
    \end{equation}

    When it comes to the difference between the increments, we have that
    \begin{align*}
        \E\left[ |(R_{(t)_\Delta}-R_{(s)_\Delta})-(R^\Delta_{(t)_\Delta}-R^\Delta_{(s)_\Delta})|\right]\leq &\E\left[\left|\int_{((s)_\Delta,(t)_\Delta]\times \R_+\times \R}y \left(\mathds 1_{\theta \leq \lambda_u} -\mathds 1_{\theta \leq \lambda^\Delta_u} \right)P(\dx u, \dx \theta, \dx y)\right|\right]\\
        \leq  &\int_{\R_+} |y|\nu(\dx y)\E \left[ \int_{((s)_\Delta,(t)_\Delta]} \left| \lambda_u- \lambda^\Delta_u \right|\dx u \right]\\
        =&{\E |Y|}\int_{((s)_\Delta,(t)_\Delta]} \E\left[\left| \lambda_u- \lambda^\Delta_u \right|\right]\dx u\\
        \leq& K \left(\int_{((s)_\Delta,(t)_\Delta]} \E\left| \lambda_{(u)_\Delta}- \lambda^\Delta_u \right|+ \E |\lambda_{(u)_\Delta}-\lambda_u| \dx u\right),
    \end{align*}
    where the first term in the integral measures the discretization error on the grid and the second term the variation between a given point and its projection on the grid.\\
    The second term in the integral can be bounded using Lemma \ref{lmm:delta}:
    \begin{equation}
    \label{ineq:projection-R}
        \int_{(s)_\Delta}^{(t)_\Delta}\E |\lambda_{(u)_\Delta}-\lambda_u| \dx u \leq \frac{K}{1-\rho_{h}}\left(  \int_0^\Delta|h(y)|\dx y + \sup_{\epsilon \in[0, \Delta]}\int_0^{T-\Delta}\left|h(y+\epsilon)-h(y)\right|\dx y\right)\left( (t)_\Delta-(s)_\Delta\right).
    \end{equation}
    The first term is upper bounded using Lemma \ref{lmm:comp-intens} 
    %{and Lemma \ref{lmm:superfluous}}
 \begin{equation}
 \label{ineq:lambda_discret-R*}
    \int_{(s)_\Delta}^{(t)_\Delta}  \E \left| \lambda_{(u)_\Delta}- \lambda^\Delta_u \right| \dx u\leq K C_R(h,\Delta) \left( (t)_\Delta-(s)_\Delta\right).
 \end{equation}

 Combining inequalities \eqref{ineq:facile-R}, \eqref{ineq:lambda_discret-R*} and \eqref{ineq:projection-R} yields the result. 
\end{proof}
\subsection{Proof of Proposition \ref{prop:martingale}}

    \begin{proof}
    Note that $(R_t- \E Y \int_0^t \lambda_s \dx s)_t$ is a continuous time $\mathcal F-$martingale and  $(\Xi_{k})_{k=0,\cdots, M}$ is a discrete time martingale. As for the discrete process, we have that 
\begin{align*}
    \Xi^\Delta_k &=R^{\Delta}_{k\Delta} - \E Y  \sum _{i=1}^k \lambda^\Delta_{i\Delta} \Delta \\
    &=\sum _{i=1}^k  \int_{((i-1)\Delta, i\Delta]}\int_{\R_+ \times \R}y\mathds 1_{\theta \leq \lambda^\Delta_{i\Delta}} P(\dx s, \dx \theta, \dx y) -\dx s \dx \theta \nu (\dx y).
\end{align*}
By construction we have that $ \int_{((i-1)\Delta, i\Delta]}\int_{\R_+ \times \R}y\mathds 1_{\theta \leq \lambda^\Delta_{i\Delta}} P(\dx s, \dx \theta, \dx y) -\dx s \dx \theta \nu (\dx y) \in \mathcal F_{i\Delta}$ and $\lambda^\Delta_{i\Delta} \in \mathcal F_{(i-1)\Delta},$ which yields the result. \\
For the upper bound, we start by combining the Cauchy Schwarz inequality with Doob's maximal inequality:
\begin{align*}
    \mathbb E \left[\max_{k=1,\cdots,M} |\Xi_{k}-\Xi^\Delta_{k}|\right ] &\leq \mathbb E \left[\max_{k=1,\cdots,M} |\Xi_{k}-\Xi^\Delta_{k}|^2\right ] ^{1/2}\\
    &\leq 2 \sqrt{\E [|\Xi_{M}-\Xi^\Delta_{M}|^2] }.
\end{align*}
Keeping in mind that $\Delta M=T$, we have that
\begin{align*}
    \Xi_M&-\Xi_M^\Delta\\=& \int _0^T \int_{\R_+\times\R} y\mathds 1_{\theta \leq \lambda_s } \left( P(\dx s, \dx \theta, \dx y) - \dx s \dx \theta \nu(\dx y)\right)  \\&- \sum _{i=1}^M  \int_{((i-1)\Delta, i\Delta]}\int_{\R_+ \times \R}y\mathds 1_{\theta \leq \lambda^\Delta_{i\Delta}} \left(P(\dx s, \dx \theta, \dx y) -\dx s \dx \theta \nu (\dx y)\right)\\
    =& \int _0^T \int_{\R_+\times \R}  \sum_{i=1}^M y\mathds 1_{(i-1)\Delta<s\leq i\Delta} (\mathds 1_{\theta \leq \lambda_s}- \mathds 1_{\theta \leq \lambda^\Delta_{i\Delta}}) (P(\dx s, \dx \theta, \dx y) -\dx s \dx \theta \nu(\dx y )).
\end{align*}
Therefore 
\begin{align*}
    \E \left[ |\Xi_{M}-\Xi^\Delta_{M}|^2\right] &= \E \left[[\Xi-\Xi^\Delta]_T\right]\\
    &=\E \int _0^T \int_{\R_+\times \R} \left( \sum_{i=1}^M y\mathds 1_{(i-1)\Delta<s\leq i\Delta} (\mathds 1_{\theta \leq \lambda_s}- \mathds 1_{\theta \leq \lambda^\Delta_{i\Delta}})\right)^2\dx s \dx \theta \nu (\dx y)\\
    &=\E \int _0^T \int_{\R_+ \times \R} \sum_{i=1}^M y^2\mathds 1^2_{(i-1)\Delta<s\leq i\Delta} (\mathds 1_{\theta \leq \lambda_s}- \mathds 1_{\theta \leq \lambda^\Delta_{i\Delta}})^2\dx s \dx \theta \nu (\dx y),\\
\end{align*}
because $\mathds 1_{(i-1)\Delta<s\leq i\Delta} \mathds 1_{(j-1)\Delta<s\leq j\Delta}=0$ if $i\neq j$. Thus 
\begin{align*}
     \E \left[ |\Xi_{M}-\Xi^\Delta_{M}|^2\right] &=\E [Y^2]\sum_{i=1}^M \int_{(i-1)\Delta}^{i\Delta} \E [|\lambda_s -\lambda^\Delta_{t_i}|]\dx t.
     \end{align*}
 By adding and subtracting $\lambda_{t_i}$ we have that, for all $s \in ((i-1)\Delta,i\Delta]$ 
    $$\E |\lambda_s-\lambda^\Delta_{\Delta i}|\leq \E |\lambda_s-\lambda_{ \Delta i}|+ \E |\lambda_{\Delta i}-\lambda^\Delta_{\Delta i}|.$$
    Following the proof of Lemma \ref{lmm:delta} and projecting on the nearest upper point of the grid (instead of the lower) we have the bound
    $$\E |\lambda_s-\lambda_{\Delta i}|\leq K\left(  \int_0^\Delta|h(y)|\dx y + \sup_{\epsilon \in[0, \Delta]}\int_0^{T-\Delta}\left|h(y+\epsilon)-h(y)\right|\dx y\right).$$
    Thus, using Inequality \eqref{ineq:lambda_discret-R} to bound $\E |\lambda_{\Delta i}-\lambda^\Delta_{\Delta i}|$ we obtain
    with the expression of $C_R(h,\Delta)$ given in Definition \ref{rem-dec-const*}
    \begin{align*}
    \E \left[ |\Xi_{M}-\Xi^\Delta_{M}|^2\right]
         &\leq K  C_R(h,\Delta) T 
         %\int_0^\Delta |h(s)| \dx s + \sup _{\epsilon \in [0,\Delta]} \int_0^{T-\Delta} |h(s+\epsilon)-h(s)| \dx s
\end{align*}
for a constant $K$ that does not depend on $T$ nor $\Delta$.
\end{proof}

    \section{Technical lemmata}
\label{sec:lemmata}
\subsection{Properties of the intensities }
    \begin{Lemma}
    \label{lmm:rec}
    Assume Assumption \ref{ass:stability_discrete} and $\Delta \in (0,T]$. Then, 
    $$\E \lambda_t^\Delta \leq \frac{\psi(0)}{1-\rho_{h,\Delta}}$$
    for any $t \in (0,T]$.
\end{Lemma}

\begin{proof}

    Since $\lambda^\Delta_t= \lambda^\Delta_{n_t \Delta}$, the inequality on the expected value of the discrete intensity will be proven by induction on $n$. We have that, for any $0\leq n \leq M$
    $$\lambda^\Delta_{n \Delta}=\psi \left(\sum_{k=1}^{n-1} h_{n_t-k}X^\Delta_k\right).$$
    In particular, for $n=0,1$
    $$\lambda^\Delta_{n\Delta} = \psi(0) \leq \frac{\psi(0)}{1-\rho_{h,\Delta}}.$$ Let $n\geq 2$ and assume that for all $j\leq n <M$
    $$\E \lambda^\Delta_{j\Delta} \leq \frac{\psi(0)}{1-\rho_{h,\Delta}}.$$
    Using the fact that $\psi$ is $L-$Lipschitz:
    \begin{align*}
        \lambda^\Delta_{(n+1)\Delta}& \leq \psi(0) + L \left|\sum_{k=1}^{n} h_{n+1-k}X^\Delta_k \right|\\
        & \leq \psi(0) + L \sum_{k=1}^{n}  |h|_{n+1-k}  X^\Delta_k.
    \end{align*}
    Since $\E X^\Delta_k= \E\left[ \E \left[ X^\Delta_k| \lambda^\Delta_{k\Delta}\right]\right ]= \E[b(Y)]\Delta \E \lambda^\Delta_{k\Delta}$ for any $k\geq 1$, we have that 
    \begin{align*}
        \E \lambda^\Delta_{(n+1)\Delta} &\leq \psi(0) + L \sum_{k=1}^{n}  |h|_{n+1-k}  \E X^\Delta_k\\
        &=\psi(0) + L \sum_{k=1}^{n}  |h|_{n+1-k}   \E[b(Y)]\Delta \E \lambda^\Delta_{k\Delta}\\
    \end{align*}
    which, according to the induction's hypothesis is bounded by:
    \begin{align*}
        \E \lambda^\Delta_{(n+1)\Delta} &\leq \psi(0) + \frac{L\psi(0)}{1-\E [b(Y)]L\sum_{k=1}^{M-1} |h|_k \Delta} \E [b(Y)]\sum_{k=1}^{n}  |h|_{n+1-k}   \Delta\\
        &\leq \psi(0) + \frac{L\psi(0)}{1-\E [b(Y)]L\sum_{k=1}^{M-1} |h|_k \Delta} \E [b(Y)]\sum_{k=1}^{M-1} |h|_{k}   \Delta
    \end{align*}
    which yields the result for $n+1\leq M$.

\end{proof}
\begin{Lemma}
\label{lmm:intensite_bornee}
     Suppose that Assumption \ref{ass:stability} is in force.
     %that $\rho_{h}:= \E [b(Y)]L\|h\|_1<1.$ 
     Let $t \in \mathbb [0,T].$ We have that
    $$\E \lambda_t \leq  \frac{\psi(0)}{1-\rho_{h}}$$
\end{Lemma}

\begin{proof}
The fact that $\lambda_t$ is integrable and $t\mapsto {\mathbb E}(\lambda_t)$ is locally bounded follow from the same lines as the proof of Theorem 1 of \cite{BM}.
    Using the fact that $\psi$ is $L-$Lipschitz we have that
    \begin{align*}
        \lambda_t&=\psi \left( \int_0^{t-} h(t-s)\dx \xi_s\right)\\
        &\leq \psi(0) +L \left|\int_0^{t-} h(t-s)\dx \xi_s \right|\\
        &\leq \psi(0) +L \int_0^t |h|(t-s) \dx \xi_s.
    \end{align*}
  
    which yields by taking the expected value
    \begin{equation}
    \label{ineq:lambda}
        \E \lambda_{t}\leq\psi(0) + \int_0^t L|h|(t-s) \E [b(Y)] \E\lambda_s \dx s. 
    \end{equation}
    Let $*$ denote the convolution operator $(f*g)(t)=\int_0^t f(s)g(t-s) \dx s$ for any integrable $f$ and $g$. Let 
    \begin{align}\label{def-S}
    S=\sum_{n\geq 1} \left(\E [b(Y)]L|h|\right)^{(n)}
    \end{align}
    where $\left(\E [b(Y)]L|h|\right)^{(n)}=\E [b(Y)]L|h|\underbrace{*\cdots*}_{n \text{ times}}\E [b(Y)]L|h|$. The stability condition $\E [b(Y)]L\|h\|_1<1$ ensures that $S$ is well defined and that $\|S\|_1 = \E [b(Y)]L\|h\|_1(1-\E [b(Y)]L\|h\|_1)^{-1}.$\\
    For a given $k\geq 0$, convoling inequality \eqref{ineq:lambda} with $L|h|^{(k)}$ yields
    $$((\E [b(Y)]L|h|)^{(k)}*\E\lambda)(t)-((\E [b(Y)]L|h|)^{(k+1)}*\E\lambda)(t)\leq \left( (\E [b(Y)]L|h|)^{(k)} * \psi(0)\right)(t)$$
    which yields after telescoping 
    $$\E \lambda_t -((\E [b(Y)]L|h|)^{(n+1)}*\E\lambda)(t)\leq \psi(0)+\sum_{k=1}^n\left( (\E [b(Y)]L|h|)^{(k)} * \psi(0)\right)(t).$$
    Since $\E \lambda_t$ is finite,  $(\E [b(Y)]L|h|)^{(n+1)}*\E\lambda$ converges to zero as $n$ tends to infinity, and finally 
    \begin{align*}
        \E \lambda_t &\leq \psi(0)+\sum_{k=1}^{M}\left( (\E [b(Y)]L|h|)^{(k)} * \psi(0)\right)(t)\\
        &=\psi(0) \left( 1+\sum_{k=1}^{M} (\E [b(Y)]L|h|)^{(k)} * 1\right)\\
        &=\psi(0) \left( 1+\frac{\E [b(Y)]L\|h\|_1}{1-\E [b(Y)]L\|h\|_1}\right)
    \end{align*}
hence the result. 

% By induction, we can obtain a bound for the expectation of the discretized process (Lemma \ref{lmm:rec}):
% $$ \E\lambda_t^\Delta \leq \frac{\psi(\mu)}{1 - L\Delta\sum_{j=1}^{+\infty}|h_j|}$$ %where we denote $\Vert h^{*}\Vert_{1} := \sum_{i=1}^{n_T - 1} \vert h_i \vert$ \\
% %$$ \E\lambda_t^\Delta \leq \frac{\psi(\mu)}{1 - L\Vert h^{*}\Vert_{1}}$$ where we denote $\Vert h^{*}\Vert_{1} := \left\sum_{i=1}^{n_T - 1} \vert h_i \vert$ \\
\end{proof}

\begin{Lemma}
    \label{lem-moment-2-discret}
   Assume that $\rho_{h,\Delta}= L \E[b(Y)]\sum_{k=1}^M |h_k| \Delta<1$  and that $\E[b(Y)^2]<\infty.$ 
     Then,
     \begin{align*}
         \sup_{t \in [0,T]} {\mathbb E}[ (\lambda_t^\Delta)^2]\leq\left (\psi (0)^2 + L^2 \E [b^2(Y)] \frac{\psi(0)}{1-\rho_{h,\Delta}}\|h\|_{\Delta,2}^2 \right)  \frac{1}{(1-\rho_{h,\Delta})^2},
     \end{align*}
    where $\|h\|_{\Delta,2}:= \left ( \sum_{j=1}^M |h_{j}|^2  \Delta\right)^{1/2}.$
\end{Lemma}
\begin{proof}
     The process $\lambda^{\Delta}$ is piecewise constant. 
    Let us recall the definition of $\lambda_{\Delta n}^{\Delta}=l_n^{\Delta}$ given in Definition  \ref{def:hawkes_discrete}~:
    \begin{align*}
        l_n^{\Delta}= \psi\left( \sum_{k=1}^{n-1} h_{n-k} X_k \right).
    \end{align*}
    Since $\psi$ is $L$  Lipschitz continuous, we have
       \begin{align*}
        l_n^{\Delta}\leq \psi(0) + L\left( \sum_{k=1}^{n-1}|h_{n-k}| X_k \right).
    \end{align*}
     Let us introduce $\tilde{X}_k =X_k - \Delta {\mathbb E}(b(Y)) l_k^{\Delta}$, then 
      \begin{align*}
        l_n^{\Delta}&\leq \psi(0) + L\left( \sum_{k=1}^{n-1}|h_{n-k}| \tilde {X}_k \right) + \sum_{k=1}^{n-1}L \Delta {\mathbb E}(b(Y)) |h_{n-k}|l_k^\Delta \nonumber \\
        &=\tilde {v}_n +   \sum_{k=1}^{n-1}L \Delta {\mathbb E}(b(Y)) |h_{n-k}|l_k^\Delta,
    \end{align*}
    where $\tilde {v}_n := \psi(0) + L\left( \sum_{k=1}^{n-1}|h_{n-k}| \tilde {X}_k \right)$. Since $l^\Delta_n$ is nonnegative, we have that 
    \begin{align}
    \label{ineq:lem-2-discret}
        l_n^{\Delta} &\leq  \tilde {v}_n + \sum_{k=0}^{n}L \Delta {\mathbb E}(b(Y)) |h_{n-k}|l_k^\Delta \nonumber\\
        &=  \tilde {v}_n + \left (L \Delta {\mathbb E}(b(Y)) |h| \ast l^ \Delta \right)_n, 
    \end{align}
    $\ast$ here being the discrete convolution operator for sequences defined on $\mathbb N$ and that take the value $0$ for $n> M$. For a given $j \in \mathbb N$ we recursively define for any $n=0, \cdots, M $:
    \begin{equation*}
        \begin{cases}
        a_n^{(0)}&=\boldsymbol{1}_{n=0}\\
        a^{(1)}_n &= a_n = L \Delta \E \left [ b(Y) \right] |h_{n}| \\
        a^{(j)}_n &= (a^{(j-1)} \ast a)_n
        \end{cases}
    \end{equation*}

For a given integer $j$ we take the convolution of Inequality \eqref{ineq:lem-2-discret} (whose right hand side is nonnegative) with the nonnegative sequence $a^{(j)}$, yielding for any $n =0,\cdots,M$
\begin{align*}
    \left ( a^{(j)}\ast l^\Delta\right)_n \leq \left(a^{(j)}\ast \tilde v \right)_n + \left (  a^{(j+1)}\ast l^\Delta\right)_n
\end{align*}
and therefore by telescoping 
\begin{align*}
    l^\Delta _n -  \left (  a^{(j+1)}\ast l^\Delta\right)_n & \leq \sum _{i=0}^ {j+1} \left ( a^{(i)}\ast \tilde v \right)_n.
\end{align*}
Since $l^\Delta_n$ is almost surely finite for any $n=0, \cdots, M$ (in fact it has a finite first moment) and since $\|a\|_1 = \left \|L \Delta \E [b(Y)]|h| \right \|_1 < 1$, we have that 
\begin{align*}
 \left (  a^{(j+1)}\ast l^\Delta\right)_n & \leq \| a^{(j+1)}\ast l^\Delta\|_1 \\
 & \leq \|a\|_1^{j+1} \|l^\Delta\|_1  \xrightarrow[\enskip j \to +\infty\enskip]{} 0.
\end{align*}
Hence, for any $n=0, \cdots, M$
$$ l^\Delta _n  \leq \tilde {v}_n + \sum _{k=0}^n A_{n-k}  \tilde {v}_{k}, $$
where $A:=\sum_{j=1}^{+\infty} a^{(j)}$ is well defined because $\|a\|_1<1$ and satisfies $\|A\|_1 \leq \frac{\|a\|_1}{1-\|a\|_1}$. Taking the square yields 
\begin{align*}
 (l_n^\Delta)^2 &\leq (\tilde{v}_n)^2 + 2 \sum _{k=0}^n A_{n-k}  \tilde {v}_{k}\tilde{v}_n + \left ( \sum _{k=0}^n A_{n-k}  \tilde {v}_{k}\right)^2 \\
 &=(\tilde{v}_n)^2 + 2 \sum _{k=0}^n A_{n-k}  \tilde {v}_{k}\tilde{v}_n +  2\sum _{0\leq j < k \leq n} A_{n-k} A_{n-j}  \tilde {v}_{k} \tilde {v}_{j}  + \sum_{k=0}^n (A_{n-k})^2 (\tilde{v}_k)^2.\\
\end{align*}
Using the definition of $\tilde v$ and the fact that $\tilde X$ is a martingale increment sequence, we have for $k \leq n $
\begin{align*}
    \E \left [ \tilde {v}_k \tilde {v}_n\right] &= \psi (0)^2 + L^2 \sum_{j=1}^k |h_{n-j}| |h_{k-j}| \E [l^\Delta_j] \Delta \E [b^2(Y)] \\ 
    &\leq  \psi (0)^2 + L^2 \E [b^2(Y)] \frac{\psi(0)}{1-\rho_{h,\Delta}}\sum_{j=1}^k |h_{n-j}| |h_{k-j}|  \Delta \\
    & \leq \psi (0)^2 + L^2 \E [b^2(Y)] \frac{\psi(0)}{1-\rho_{h,\Delta}}\left(\sum_{j=1}^k |h_{k-j}|^2 \Delta  \right)^{1/2} \left ( \sum_{j=1}^n |h_{n-j}|^2  \Delta\right)^{1/2} \\ 
    & \leq \psi (0)^2 + L^2 \E [b^2(Y)] \frac{\psi(0)}{1-\rho_{h,\Delta}}\|h\|_{\Delta,2}^2, \\ 
\end{align*}
where $\|h\|_{\Delta,2}:= \left ( \sum_{j=1}^M |h_{j}|^2  \Delta\right)^{1/2}.$ Since $\|A\|_1 \leq \frac{\rho_{h,\Delta}}{1-\rho_{h,\Delta}}$ we have 
\begin{align*}
    \E \left [ (l^\Delta_n)^2 \right] & \leq \left (\psi (0)^2 + L^2 \E [b^2(Y)] \frac{\psi(0)}{1-\rho_{h,\Delta}}\|h\|_{\Delta,2}^2 \right)  \left ( 1+ 2 \frac{\rho_{h,\Delta}}{1-\rho_{h,\Delta}} +   \left( \frac{\rho_{h,\Delta}}{1-\rho_{h,\Delta}}\right)^2 \right)\\
    &=\left (\psi (0)^2 + L^2 \E [b^2(Y)] \frac{\psi(0)}{1-\rho_{h,\Delta}}\|h\|_{\Delta,2}^2 \right)  \left ( 1+ \frac{\rho_{h,\Delta}}{1-\rho_{h,\Delta}} \right)^2 \\
    & \leq \left (\psi (0)^2 + L^2 \E [b^2(Y)] \frac{\psi(0)}{1-\rho_{h,\Delta}}\|h\|_{\Delta,2}^2 \right)  \frac{1}{(1-\rho_{h,\Delta})^2} . \\
\end{align*}
 
\end{proof}
 \begin{Lemma}\label{lem-moment-2}
   Under Assumption \ref{ass:stability} and the $\E b(Y)^2 <\infty$
     \begin{align*}
         \sup_{t \in [0,T]} {\mathbb E}[ \lambda_t^2]\leq\left (\psi(0)^2 + L^2 \E[b(Y)^2] \frac{\psi(0)}{1-\rho_{h}}\|h\|^2_2 \right)\frac{1}{(1-\rho_{h})^2}.
     \end{align*}
 \end{Lemma}

\begin{Remark}
    Unlike Theorem 2.4 of \cite{hillairet2023explicit} in which the authors give an exact expression of $\E [\lambda_t^2]$, we provide here an upper bound on that quantity. Our result has the advantage of being explicit in the parameters of the Hawkes process and of illustrating that the second moment is also bounded in $t$ when the stability condition is verified. 
\end{Remark}
\begin{proof}
Recall that from Definition \ref{def:hawkes_continuous}
\begin{align*}
    \lambda_t =\psi \left( \int_{[0,t) \times \R_+ \times \R} h(t-s) \mathds 1_{\theta \leq \lambda _s} b(y)P(\dx s, \dx \theta, \dx y) \right).
\end{align*}
Since $\Psi$ is Lipschitz continuous, 
we have 
\begin{align*}
    \lambda_t &\leq \psi (0)+ L\int_{[0,t) \times \R_+ \times \R} |h(t-s)| \mathds 1_{\theta \leq \lambda _s} b(y)P(\dx s, \dx \theta, \dx y)\\
    &\leq \tilde{M}_t +L \E[b(Y)]\int_0^t |h(t-s)| \lambda_s ds,
\end{align*}
where
\begin{align*}
    \tilde{M}_t = \psi (0)+ L\int_{[0,t) \times \R_+ \times \R} |h(t-s)| \mathds 1_{\theta \leq \lambda _s} b(y)\left(P(\dx s, \dx \theta, \dx y) - \dx s \dx \theta \nu(\dx y) \right).
\end{align*}
%\textcolor{magenta}{assertion a verifier} \textcolor{blue}{A-t-on vraiment besoin de verifier? La stabilité suffit a mon avis. N'oublions pas que ce sont des trajectoires continues par morceaux.}

%Almost surely, the maps $t \mapsto \lambda_t$ and $t \mapsto \tilde{M}_t$ are Borel and locally bounded functions.
Using the same lines as the proof of Lemma \ref{lmm:intensite_bornee} or Lemma 3 of \cite{BACRY20132475},
\begin{align*}
   \lambda_t \leq  \tilde{M}_t + \int_0^t S(t-s) \tilde{M}_s \dx s
\end{align*}
where $S $ is defined in \eqref{def-S}.
Then, 
\begin{align*}
   \lambda_t^2 &\leq  \tilde{M}_t^2 + 2\int_0^t S(t-s) \tilde{M}_s \tilde{M}_t \dx s + \left(\int_0^t S(t-s) \tilde{M}_s \dx s\right)^2\\
   &\leq  \tilde{M}_t^2 + 2\int_0^t S(t-s) \tilde{M}_s \tilde{M}_t \dx s + \int_{[0,t]^2} S(t-s) S(t-u)\tilde{M}_s \tilde{M}_u \dx s \dx u.
\end{align*}
According to the definition of $\tilde{M}_t , $ the fact that $h$ is bounded, $\rho_{h}=L\E[b(Y)] \|h\|_1<1$ and Lemma \ref{lmm:delta}
\begin{align*}
    \E \left[ \tilde{M}_s\tilde{M}_u \right]&= \psi(0)^2 + L^2 \int_0^{\min(u,s) } \E[b(Y)^2] |h(s-r)||h(u-r)| \E[\lambda_r]\dx r \\
    &{\leq \psi(0)^2 + L^2 \E[b(Y)^2] \frac{\psi(0)}{1-\rho_{h}}\int_0^{\min(u,s) }  |h(s-r)||h(u-r)| \dx r}\\
    &{\leq \psi(0)^2 + L^2 \E[b(Y)^2] \frac{\psi(0)}{1-\rho_{h}}\left(\int_0^{s }  h^2(s-r) \dx r\right)^{1/2} \left(\int_0^{u }  h^2(u-r) \dx r\right)^{1/2}}\\
    &{= \psi(0)^2 + L^2 \E[b(Y)^2] \frac{\psi(0)}{1-\rho_{h}}\|h\|^2_2}.
    %&\leq \psi(0)^2 +L\|h\|_{\infty}\frac{\psi(0)}{1-\rho_{h}}\frac{\E [b(Y)^2]}{\E[b(Y)]^2 }.
\end{align*}
%\textcolor{magenta}{nous on majore : si on gardait $\int_0^{\min(u,s) } \E[b(Y)^2] |h(s-r)||h(u-r)| \E[\lambda_r]\dx r$ je crois qu'on aurait quelque chose comme caroline et Anthony}
Since $\|S\|_1 ={\rho_{h}}[1-\rho_{h}]^{-1}$ we have 
\begin{align*}
    \E[\lambda_t^2 ] &\leq
     \left (\psi(0)^2 + L^2 \E[b(Y)^2] \frac{\psi(0)}{1-\rho_{h}}\|h\|^2_2 \right)\left( 1 + \frac{2\rho_{h}}{1-\rho_{h}} + \frac{\rho_{h}}{(1-\rho_{h})^2}\right)\\
    &=\left (\psi(0)^2 + L^2 \E[b(Y)^2] \frac{\psi(0)}{1-\rho_{h}}\|h\|^2_2 \right)\frac{1}{(1-\rho_{h})^2}.
\end{align*}
\end{proof}

    \subsection{Estimation on the modulus of continuity of  Compound Poisson processes}
  
    \begin{Lemma}
    \label{lmm:module_poisson}
        Let $N$ be a Poisson process of intensity $I$ on $[0,1]$ and $R_t=\sum_{k=0}^{N_t}|Y_k|$ where $(Y_k)_k$ is a sequence of independent identically distributed random variables with common distribution $\nu$ such that $\int_{{\mathbb R}} |y|\nu(\dx y)<+\infty.$ Then its average modulus of continuity in ${\mathbb D}([0,T],{\mathbb R})$   is bounded by 
        $$\E \omega'_R(\Delta,[0,T]) \leq \E |Y_1|IT \frac{\Delta}{T} (2+4IT)=2\E |Y_1|I \Delta (1+2IT).$$

    \end{Lemma}

    \begin{proof}
First, we assume that $T=1.$ Let $\tau_1, \tau_2, \ldots$ denote the arrival times of the Poisson process and $S_1, S_2,\ldots$ be the inter-arrival times. Using the law of total probability we have that
        % $$\E \omega'_N(\Delta)  = \E [\omega'_N(\Delta)|A] \mathbb P [A] + \E [\omega'_N(\Delta)|A^c] \mathbb P [A^c], $$
        % where $A=\{T- {\tau_{N_T}} \geq \Delta\}$, $\tau_{N_T}$ being the time of the last jump before $T$. Since $T-\tau_{N_T}$ is the age of the Poisson process, it follows a truncated exponential distribution which yields 
        %  $$\E \omega'_N(\Delta)  = \E [\omega'_N(\Delta)|A] e^{-\Delta \|\psi\|_\infty}+ \E [\omega'_N(\Delta)|A^c](1-e^{-\Delta \|\psi\|_\infty}). $$
        %  On the event $A^c$, one can choose the $\Delta-$sparse partition (in the infimum for the modulus of continuity) to be 
         \begin{align*}
             \E \omega'_R(\Delta,[0,1])  &= \sum_{n=1}^{+\infty}\E [\omega'_R(\Delta)|N_T=n] \mathbb P [N_T=n] \\
             &= \sum_{n=1}^{+\infty}\E_n[\omega'_R(\Delta)] \frac{(I)^n}{n!}e^{-I}. \\
         \end{align*}
         Knowing that $N_T=n$, the distribution of the arrival times is that of the order statistics of $n$ uniform random variables on $[0,1]$, that is of density
         $p(\tau_1=t_1,\cdots, \tau_n=t_n)=\mathds 1_{0 \leq t_1  \leq \cdots \leq t_n \leq 1}n!$. Using an affine change of variables we obtain a similar formula for the density of the inter-arrival times 
         $$p(S_1=s_1,\cdots, S_n=s_n)=\mathds 1_{0\leq s_1+\cdots+ s_n   \leq 1} n!\prod_{i=1}^n \mathds 1_{s_i \geq 0}.$$
         Two scenarii are possible
         \begin{enumerate}
             \item All of the inter-arrival times are larger than $\Delta$ (only possible if $n\Delta\leq T$). In this case we only have two possibilities:
             \begin{itemize}
                 \item $\omega '_R(\Delta,[0,1])=|Y_n|$ if the last arrival time $\tau_n$ is at a distance less than $\Delta$ from $1$.
                 \item $\omega '_R(\Delta,[0,1])=0$ otherwise.
             \end{itemize}
             \item At least one interarrival time $S_i$ for $i=1,\ldots,n$ is less than $\Delta$. In this case, the worst case scenario is to have all of the jumps in one interval of size at most $\Delta$, yielding 
             $$\omega '_R(\Delta,[0,1])\leq \sum_{k=1}^n |Y_k|.$$
         \end{enumerate}
         Hence we have that 
        \begin{align*}
            \E _n\omega '_R(\Delta,[0,1])&\leq \E_n |Y_n|  \mathbb P_n [A_1] + \sum_{k=1}^n \E_n|Y_k| \mathbb P_n[A_2] \\
            &=\E |Y_1|  \mathbb P_n [A_1] + n \E|Y_1| \mathbb P_n[A_2],
        \end{align*}
        where $A_1=\{S_i\geq \Delta, \forall i=1,\ldots, n \text{ and } 1-{\tau_n} <\Delta\}$ and $A_2=\{ \exists i \in [1,n], S_i\leq \Delta\}$.\\
        Keeping in mind that $\tau_n$ is the maximum of $n$ uniform iid variables on $[0,1]$ we have that 
        \begin{align*}
             \mathbb P_n [A_1] &\leq \mathbb P_n[1-\Delta  <{\tau_n}]\\
             & =1-  \mathbb P _n [1-\Delta  >{\tau_n}]\\
             &=1-  \mathbb P _n [1-\Delta  >\mathcal{U} [0,1]]^n \quad \text{because $\tau_n= \max_{i=1,\cdots,n} U_i $}\\
             &= 1- \left(1-\Delta \right)^n\\
             &\leq {(n\Delta)\wedge 1}.\\
           %  &\leq 2\Delta.
        \end{align*}
       
     %  \textcolor{magenta}{ In the last inequality,  we have used the fact that for $0<a<1/2,$  
      %  $$1-(1-a)^n= a (1+...+a^{n-1})\leq \frac{a}{1-a} \leq 2a.$$
        %}
        For the event $A_2$ we have that 
        \begin{align*}
            \mathbb P_n[A_2] &= \mathbb P_n [\min_{1\leq i \leq n} (S_i) \leq \Delta]\\
            &= \int_{\R_+^n} \mathds {1}_{\min_{1\leq i \leq n} (s_i) \leq \Delta}\mathds 1_{0\leq s_1+\cdots+ s_n   \leq 1} n!\dx s_1 \ldots \dx s_n \\
            &\leq \sum _{i=1}^n \int _{0\leq s_1+\cdots+ s_n   \leq 1} \mathds 1_{s_i \leq \Delta } n! \dx s_1 \ldots \dx s_n\\
            &=n \int _{0\leq s_1+\cdots+ s_n   \leq 1} \mathds 1_{s_1 \leq \Delta } n! \dx s_1 \ldots \dx s_n\\
            &\leq n n! \int_0^\Delta \dx s_1 \int_{0 \leq s_2 +\cdots + s_n \leq 1}\dx s_2 \ldots \dx s_n\\
            &= n n! \Delta \frac{1}{(n-1)!}\\
            &= n^2 \Delta .
            %\leq \textcolor{magenta}{n^2 \Delta \wedge 1}.
        \end{align*}
  
    Therefore 
    $$\mathbb E_n \omega '_R(\Delta,[0,1])\leq \left[{n} \E  |Y_1| + {n \E |Y_1|} n^2\right] \Delta.$$
    And by averaging over the Poisson variable:
    \begin{align}
    \label{ineq:modulus}
        \mathbb E [\omega'_R(\Delta,[0,1])]  &\leq  \E |Y_1| \Delta (2 + 3 I  + I) I  \nonumber\\
        &=\E |Y_1| \Delta I (2 + 3 I + I).
    \end{align}
 Now we define the time scaled process $R^T_v=R_{vT}$, where $v\in [0,T]$. The process $R^T$ is also a compound Poisson process of intensity $IT$. Using the fact that $\Delta-$sparse subdivisions of $[0,T]$ are exactly the $\frac{\Delta}{T}-$sparse subdivisions of $[0,1]$ multiplied by $T$, we have that 
   \begin{align*}
       \omega'_{R}(\Delta, [0,T]) &= \inf_{\Delta-\text{sparse}}\max _{1\leq i \leq K} \sup_{u,v \in [t_{i-1},t_{i})}|R_u-R_v|\\
       &=\inf_{\Delta-\text{sparse}}\max _{1\leq i \leq K} R_{t_i-}-  R_{t_{i-1}}\\
       &=\inf_{\Delta-\text{sparse}}\max _{1\leq i \leq K} R_{T\frac{t_i-}{T}}-  R_{T\frac{t_{i-1}}{T}}\\
       &=\inf_{\frac{\Delta}{T}-\text{sparse}}\max _{1\leq i \leq K} R^{T}_{s_i-}-  R^{T}_{s_{i-1}}=\omega'_{R^T} \left(\frac{\Delta}{T}, [0,1] \right).\\
   \end{align*}
We now take the expected value and use the upper bound \eqref{ineq:modulus} with intensity $IT$, time step $\frac{\Delta}{T}$ and time horizon $1$ to obtain 
$$\E \omega'_R(\Delta,[0,T]) \leq \E |Y_1|IT \frac{\Delta}{T} (2+4IT)=2\E |Y_1|I \Delta (1+2IT).$$

  \end{proof}

    \begin{Lemma}\label{lem-indic-sob}
        Let $p\geq 1,$ $\eta \in]0, p^{-1}]$ $0 <a <b,$  then ${\mathbf 1}_{[a,b]} \in W^{\eta,p}_T\cap I_{0^+}^{\eta}(L^1([0,T]).$
    \end{Lemma}
    \begin{proof}
        It is clear that ${\mathbf 1}_{[a,b]}\in L^p({\mathbb R}^+).$

        Let $s<t$ then $\left|{\mathbf 1}_{[a,b]}(t) -{\mathbf 1}_{[a,b]}(s) \right|= {\mathbf 1}_{[a,b]}(t){\mathbf 1}_{]0,a[}(s) + {\mathbf 1}_{[a,b]}(s){\mathbf 1}_{]b,+\infty[}(s) $ thus
        \begin{align*}
            \|{\mathbf 1}_{[a,b]}\|_{W^{\eta,p}_T}^p= (b-a) + \frac{2}{p\eta( 1-p\eta)} 
            [2(b-a)^{1-p\eta}-b^{1-p\eta} +a^{1-p\eta}+|T-b|^{1-p\eta} -|T-a|^{1-p\eta}]
        \end{align*}

        Moreover,
        For $a>0,$
        \begin{align*}
            {\mathbf 1}_{[a,+\infty]}= I_{0^+}^{\eta}(g_{a,\eta})
        \end{align*}
        where   $g_a(t)= \frac{\Gamma(\eta)}{\int_0^1 (1-u)^{\eta-1}u^{-\eta} \dx u}(t-a)_+^{-\eta},~~t\geq 0.$ 
    \end{proof}

\subsection{Estimation for finite $p$ variation kernels}

We now give a more exploitable bound instead of $C_R(\Delta)$ for a class of kernels.
   \begin{Lemma}\label{lmn-p-var-majo}

       Let $T>0$ and $p\geq 1$. Assume that $h$ is of bounded $p$-variation on $[0,T].$ Then there exists a constant $K$, independent from $T,$  $h$ and $\Delta$ such that 
       \begin{align*}
           &\int_0^{\Delta}|h(y) |\dx y \leq \Delta \|h\|_{\infty},\\
           & \sup_{\varepsilon \leq \Delta}\int_0^T |h(t+\epsilon)-h(t)| \dx t \leq   K \|h\|_{p-var,T} \left(T ^{\frac{p-1}{p}}\Delta ^{\frac{1}{p}} +\Delta \right),\\
           &\int_0^T |h(y)-h\left((y)_\Delta\right)| \leq K \|h\|_{p-var,T} \left(T ^{\frac{p-1}{p}}\Delta ^{\frac{1}{p}} +\Delta \right)
       \end{align*}
       where $\|h\|_{p-var,T}$ is the $p-$variation semi-norm of $h$ on $[0,T]$ and $\|h\|_{\infty,T}=\sup_{0\leq t \leq T} |h(t)|.$
       Moreover  for $\Delta$ small enough Assumption \ref{ass:stability_discrete} is fulfilled. 
   \end{Lemma}
\begin{proof}
    We start the proof by providing an upper bound on the modulus of continuity of the shift operator in $L_1$, along the lines of Lemma A.1 in \cite{MR3443431}. Let $0<\epsilon \leq \Delta$ 
    \begin{align*}
        \int_0^T |h(t+\epsilon)-h(t)| \dx t&= \sum_{j=1}^{\lfloor T/\epsilon \rfloor+1} \int_{(j-1)\epsilon}^{j \epsilon \wedge T} |h(t+\epsilon)-h(t)| \dx t \\
        &=\sum_{j=1}^{\lfloor T/\epsilon \rfloor+1} \int_{0}^{\epsilon} |h((t+j\epsilon)\wedge T)-h(t+(j-1)\epsilon)| \dx t\\
        &=\int_{0}^{\epsilon} \sum_{j=1}^{\lfloor T/\epsilon \rfloor+1} |h((t+j\epsilon)\wedge T)-h(t+(j-1)\epsilon)| \dx t.\\
        \end{align*}
        Using Hölder's inequality, we have for $t\in[0,\epsilon]$ and $\frac{1}{q}=1-\frac{1}{p}$:
        \begin{align*}
            \sum_{j=1}^{\lfloor T/\epsilon \rfloor+1} |h((t+j\epsilon)\wedge T)&-h(t+(j-1)\epsilon)|\\  &\leq \left(\sum_{j=1}^{\lfloor T/\epsilon \rfloor+1} |h((t+j\epsilon)\wedge T)-h(t+(j-1)\epsilon)|^p \right)^{1/p} \left(\sum_{j=1}^{\lfloor T/\epsilon \rfloor+1}1\right)^{1/q}\\
            &\leq \|h\|_{p-var,T} \left( \lfloor T/\epsilon \rfloor+1\right)^{\frac{p-1}{p}}\\
            &\leq K \|h\|_{p-var,T} \left( \left(\frac{T}{\epsilon}\right) ^{\frac{p-1}{p}} +1\right).
        \end{align*}
        And by integrating from $0$ to $\epsilon$ we get
        $$\int_0^T |h(t+\epsilon)-h(t)| \dx t \leq K \|h\|_{p-var,T} \left(T ^{\frac{p-1}{p}}\epsilon ^{\frac{1}{p}} +\epsilon\right),$$
        which by taking the supremum of $\epsilon$ between $0$ and $\Delta$ yields
        $$\sup_{\epsilon \in [0,T]} \int_0^T |h(t+\epsilon)-h(t)| \dx t \leq K \|h\|_{p-var,T} \left(T ^{\frac{p-1}{p}}\Delta ^{\frac{1}{p}} +\Delta \right).$$
        In a similar fashion, we also show that 
        $$\int_0^T |h(y)-h\left((y)_\Delta\right)| \leq K \|h\|_{p-var,T} \left(T ^{\frac{p-1}{p}}\Delta ^{\frac{1}{p}} +\Delta \right).$$
        
        Indeed, 
        \begin{align*}
            \int_0^T |h(y)-&h((y)_{\Delta})|\dx y\\ &=\sum_{k=1}^{M}\int_{t_{k-1}}^{t_k } |h(y) - h((y)_{\Delta})|\dx y\\
            &=\int_0^{\Delta} \sum_{k=1}^{M} |h( (t_{k-1}+r)\wedge T) -h(t_{k-1})| \dx r\\
            &\leq \int_0^{\Delta} \sum_{k=1}^{M} |h( (t_{k-1}+r)\wedge T) -h(t_{k-1})| + | h(t_{k} \wedge T) - h( (t_{k-1}+r)\wedge T)| \dx r.
            \end{align*}
            Thus, 
            we also show that 
        $$\int_0^T |h(y)-h\left((y)_\Delta\right)| \leq K \|h\|_{p-var,T} \left(T ^{\frac{p-1}{p}}\Delta ^{\frac{1}{p}} +\Delta \right).$$
        
           In a similar fashion, we can show that 
        \begin{equation}
            \label{ineq:module_shift_2}
            \int_0^{T-\Delta} |h((t)_\Delta+\Delta)-h(t)| \dx t \leq \|h\|_{p-var,T}  T^{\frac{p-1}{p}}\Delta ^{\frac{1}{p}}.
        \end{equation}
        This means that, $\lim_{\Delta \to 0} \int_0^{T-\Delta} |h((t)_\Delta+\Delta)-h(t)| \dx t =0$ and hence, thanks to Inequality \ref{ineq: hyp_superflue} we have that Assumption \ref{ass:stability_discrete} is in force. 
        
\end{proof}

\appendix
\section*{Funding}
This work was supported by the ANR EDDA Project-ANR-20-IADJ-0003. Mahmoud Khabou acknowledges support from EPSRC NeST Programme grant EP/X002195/1.

\section*{Competing interests}
There were no competing interests to declare which arose during the preparation or publication process of this article.

% \section{Example Appendix Section}
% \label{app1}

% Appendix text.

% %% For citations use: 
% %%       \cite{<label>} ==> [1]

% %%
% Example citation, See \cite{lamport94}.

%% If you have bib database file and want bibtex to generate the
%% bibitems, please use
%%
  \bibliographystyle{elsarticle-num} 
  \bibliography{biblio}

%% else use the following coding to input the bibitems directly in the
%% TeX file.

%% Refer following link for more details about bibliography and citations.
%% https://en.wikibooks.org/wiki/LaTeX/Bibliography_Management

% \begin{thebibliography}{00}

% %% For numbered reference style
% %% \bibitem{label}
% %% Text of bibliographic item

% \bibitem{lamport94}
%   Leslie Lamport,
%   \textit{\LaTeX: a document preparation system},
%   Addison Wesley, Massachusetts,
%   2nd edition,
%   1994.

% \end{thebibliography}
\end{document}